%% file: paper_prequel.tex
\documentclass{siamart250211}
\usepackage{amssymb,amsmath}
\usepackage{tikz}
\usepackage{pgfplots,pgfplotstable}
\usepackage[hypcap=true]{subcaption}
\usepackage{xcolor}

\pgfplotsset{compat=1.18}
\newsiamthm{claim}{Claim}

\begin{document}

\title{Bounds-constrained finite element approximation of time-dependent partial differential equations\thanks{Submitted to the editors DATE
\funding{This work is supported by the U.S. National Science Foundation grant \#2410736.}}}

\author{Robert C. Kirby\thanks{Department of Mathematics, Baylor University, Waco, TX
(\email{robert\_kirby@baylor.edu}).}
\and John D. Stephens\thanks{Department of Mathematics, Baylor University, Waco, TX (\email{john\_stephens2@baylor.edu})}}

\maketitle

\newcommand{\BibTeX}{{\scshape Bib}\TeX\xspace}
\begin{abstract}
  \input{abstract.tex}
\end{abstract}

\begin{keywords}
Runge-Kutta method, multistep method, finite element method, variational inequality, bounds constraints
\end{keywords}

\begin{MSCcodes}
65M60, 65L06
\end{MSCcodes}

\input{body_prequel.tex}

\section*{Acknowledgements}

The authors would like to thank Ahsan Ali for his help implementing the $\ell$-AIR algorithm.

\bibliographystyle{plain}
\bibliography{paper}

\end{document}

%% file: abstract.tex
Finite elements provide accurate and efficient methods for the numerical solution of partial differential equations by means of restricting variational problems to 
finite-dimensional approximating spaces. However, they do not, in general, guarantee enforcement of bounds constraints inherent in the original problem.
We propose two approaches to enforcing bounds constraints for time-dependent problems.  First, we propose general projective methods which result from a systematic modification of any abstract timestepping scheme.  Second, we present a monolithic technique for which we take a modified formulation of implicit single-stage Runge-Kutta methods and general implicit multistep methods as prototypical examples. By solving a constrained optimization problem, 
we are able to ensure that the bounds constraints are enforced by the approximate solution at the discrete time levels, obtaining (formally) high order methods both in space and time.
Numerical examples for the linear heat and advection equations and nonlinear Allen-Cahn equation are given.

%% file: body_prequel.tex

\section{Introduction}\label{sec:introduction}
Preserving bounds constraints such as positivity in the numerical solution of partial differential equations (PDEs) faces many challenges.
Even ordinary differential equations (ODEs) present complexities.
For example, Bolley and Crouzeix have shown that linear single-step methods that unconditionally preserve positivity cannot exceed first order accuracy~\cite{bolley1978conservation}.
The literature contains many attempts to circumvent this barrier via more general schemes.
For example, inadmissible values may be clipped according to the bounds~\cite{Shampine1986}. 
For production-destruction systems, Patankar-type methods have been developed to ensure positivity~\cite{Bruggeman2007, Burchard2003, Huang2019, Kopecz2018, Patankar1980}.
The contractivity and positivity of diagonally split Runge-Kutta methods have been studied as well~\cite{Bellen1994}.
Additional attempts to break the order barrier of Bolley and Crouzeix can be found in~\cite{Blanes2022}. 

However, such schemes for ODEs are only relevant to PDEs when the spatial discretization preserves bounds constraints.
The famous result of Godunov~\cite{godunov1959finite} shows linear single step monotone schemes for the model convection equation are limited to first order accuracy.
Nochetto and Wahlbin~\cite{nochetto2002positivity} show that positivity-preserving linear operators on finite element spaces can only reproduce linear polynomials, limiting the accuracy of linear methods to second order.
In light of these facts, it is prudent to study the spatial and temporal discretization of bounds-satisfying PDE simultaneously.
As the order barriers apply to linear methods, it seems that some application of \emph{nonlinear} methods may be in order.

Here, we formulate two approaches for preserving bounds constraints in the discretization of time-dependent PDEs.
First, we can postprocess some fully discrete scheme (for example, finite elements in space and Runge--Kutta or multistep in time) by means of a nonlinear projection into a bounds-satisfying set.  We refer to these schemes as \emph{projective}.
In contrast, we also propose \emph{monolithic} schemes.
Here, we replace the spatial variational equation defining the new value $y_{n+1}$ with a variational inequality that enforces the constraints.

In~\cite{chang2017variational}, variational inequalities are used to enforce bounds constraints for piecewise linear discretizations of  advection-diffusion equations, and in~\cite{kirby2024variational}, this idea is extended to provide uniformly-constrained finite element approximations for steady-state 
problems with elements of arbitrary order.
The analysis there adapts Falk's analysis for variational inequalities~\cite{falk1974error} to the special case where the continuous problem is a variational equation.
Similar techniques providing nodally bounds-preserving methods are given in~\cite{barrenechea2024nodally} for steady-state problems, and~\cite{amiri2024nodally} for convection-diffusion-reaction equations. 
In these works, a stabilized finite element approach is devised which is equivalent to a variational inequality over a discrete feasible set. In~\cite{barrenechea2025nodally}, the drift-diffusion equation is discretized in time using the implicit Euler method and a discontinuous Galerkin spatial discretization with a variational inequality.

An important alternative approach to bounds enforcement is the class of proximal Galerkin methods~\cite{keith2023proximal}.
Here, the original problem is replaced with a sequence of regularized and reformulated problems, each of which preserves bounds.
Although these methods are not formulated cleanly as a single variational problem, they support high orders of approximation and provide scalable, mesh-independent convergence rates. 
Additional techniques which do not rely directly on the theory of variational inequalities include a convex optimization-based filtering approach, introduced in~\cite{zala2023convex}, where the numerical solution 
is post-processed using a sample-and-correct approach; after each time step, the numerical solution is sampled at points throughout the domain, and the filter is then used on any cells 
which exhibit bounds violations.
Finally, the work in~\cite{nusslein2021positivity} poses a linear program to find weights for Runge--Kutta stages that will preserve positivity and the order conditions of the underlying method.

Invariant-domain-preserving discretizations~\cite{ern2022invariant} capture many deep physical properties of conservation laws.
However, these techniques depend quite strongly on the particular structure of the equations being studied and may not be applicable beyond those hyperbolic problems.
The methods we develop here typically do not capture as much structure as those methods, but they are readily formulated for a much more general class of PDEs.

The rest of the paper is organized as follows: 
in Section~\ref{sec:problemSetting} we introduce the problem setting and model problems.  We develop our two approaches to preserving bounds in Section~\ref{sec:methodDev}.
Section~\ref{sec:approxTheory} contains some comments on analysis.
After presenting numerical results in Section~\ref{sec:examples}, we offer concluding remarks and directions for future research in in Section~\ref{sec:conclusion}.


\section{Problem Setting}\label{sec:problemSetting}

Let $\mathcal{V}$ and $\mathcal{H}$ be Hilbert spaces with $\mathcal{V}$ compactly embedded in $\mathcal{H}$. Let $(\cdot, \cdot)$ be the inner product on $\mathcal{H}$, and $\langle \cdot, \cdot \rangle$ the duality pairing on $\mathcal{V}'\times \mathcal{V}$. 
For each time $t>0$, let $F_t:\mathcal{V}\rightarrow \mathcal{V}'$ be a bounded mapping and consider the 
abstract evolution equation seeking some differentiable $y:\mathbb{R}^+\rightarrow \mathcal{V}$ such that 
\begin{equation}\label{eq:model_problem}
(y', v) = \langle F_t(y), v\rangle \qquad \forall v\in \mathcal{V},
\end{equation}
together with an appropriate initial condition 
\begin{equation}
  y(0) = y_0\in \mathcal{V}.
\end{equation}
We further suppose that the solution $y(t)$ is known to lie in some closed and convex set $\mathcal{K} \subset \mathcal{V}$ for all time, such as the set of nonnegative functions.

Standard Galerkin discretization of~\eqref{eq:model_problem} follows by restricting the problem to some finite-dimensional $\mathcal{V}_h \subset \mathcal{V}$ and evolving the resulting system of ODEs.
However, we can introduce a slightly more general framework that accounts for discontinuous Galerkin and other nonconforming schemes.  We can let $\mathcal{V}_h \subset \mathcal{H}$ instead of $\mathcal{V}$, let $\langle \cdot , \cdot \rangle$ be the $\mathcal{H}^\prime,\mathcal{H}$ duality, and consider a discrete evolution equation
\begin{equation}\label{eq:discmodel_problem}
(y', v) = \langle F_{h,t}(y), v\rangle \qquad \forall v\in \mathcal{V}_h.
\end{equation}
Here, $F_h$ is some possibly mesh-dependent approximation to the original $F$, for example, including jump terms arising in discontinuous Galerkin methods.
Taking $\mathcal{V}_h \subset \mathcal{V}$ and $F_{h,t} = F_{t}$ recovers the standard conforming method.

\subsection{Model Problems}
To fix notation for our model PDE problems,
we take $L^2(\Omega)$ to be the standard space of square-integrable functions over some domain $\Omega \in \mathbb{R}^d$ and $H^1(\Omega)$ its subspace of functions with square-integrable weak derivatives of order 1.  Additionally, $H^1_0(\Omega) \subset H^1(\Omega)$ contains only functions vanishing on $\partial \Omega$.

Our first PDE to consider is the heat equation, posed on $\Omega \subset \mathbb{R}^d$ for $1 \leq d \leq 3$.
Let  $f: [0, T] \rightarrow L^2(\Omega)$. 
We seek $u: [0, T] \rightarrow H^1(\Omega)$ such that 
\begin{equation}\label{eq:heat_general}
    \left( u', v \right) + \left( \nabla u, \nabla v \right) = \left(f, v \right)
\end{equation}
for all $v \in H_0^1(\Omega)$ and $0 \leq t \leq T$.  We close the problem with an initial condition
\begin{equation}
  u(0) = u_0 \in H^1(\Omega),
\end{equation}
and the Dirichlet boundary condition
\begin{equation}\label{eq:heat_bc}
u(t)|_{\partial \Omega} = g(t) \qquad \text{for } t\in [0, T].
\end{equation}

Next, we consider an advective equation. Let $\Omega \subset \mathbb{R}^d$ for $1\leq d\leq 3$. Let $\Gamma = \partial \Omega$, and $\mathbf{n}$ denote the unit outward normal vector of $\Omega$.  Let $\mathbf{v}$ be a prescribed (possibly time dependent) divergence-free vector field.
Seek $u:[0, T]\rightarrow L^2(\Omega)$ such that 
\begin{equation}\label{eq:advec_eq}
\frac{\partial u}{\partial t} + \nabla \cdot (\mathbf{v}u) = 0.
\end{equation}
We close the system by imposing the inflow boundary condition
\begin{equation}\label{eq:advec_bc}
  u|_{\Gamma_\text{inflow}}(t) = u_\text{in}(x, t) \qquad \text{ for } t\geq 0,
\end{equation}
where $\Gamma_\text{inflow}(t) = \{x\in \Gamma : \mathbf{v}(x, t)\cdot \mathbf{n} < 0\}$, 
and the initial condition
\begin{equation}\label{eq:advec_ic}
  u(0) = u_0\in L^2(\Omega).
\end{equation}
We utilize a discontinuous Galerkin approximation with an upwind flux. The weak form is standard and will be discussed in Section~\ref{sec:examples}.

Finally, we consider the Allen-Cahn equation~\cite{allen1979microscopic} with logarithmic Flory-Huggins potential~\cite{flory1942thermodynamics, huggins1941solutions}. Consider the potential function
\begin{equation}
  F(s) = \frac{\theta_0}{2}\left[(1 + s)\ln(1 + s) + (1 - s)\ln(1 - s)\right] - \frac{\theta_c}{2}s^2,
\end{equation}
where $\theta_0 < \theta_c$ are positive constants. The Allen-Cahn equation seeks $u:[0, T]\rightarrow H^1(\Omega)$ such that
\begin{equation}\label{eq:allen_cahn}
  (u_t, v) + (\nabla u, \nabla v) + (f(u), v) = 0 \qquad \forall v\in H^1(\Omega),
\end{equation}
where 
\begin{equation}
  f(u) = F'(u)
\end{equation}
is the chemical potential. We close the system by imposing periodic boundary conditions. Physical and mathematical constraints require that $-1 < u < 1$ for all positive times, providing an interesting nonlinear problem 
on which to test our methods.

\subsection{Representing Bounds Constrained Polynomials}

We let $\mathcal{V}_h$ denote the finite-dimensional space, typically a subspace of $H^1(\Omega)$ or $L^2(\Omega)$.   
Although a Galerkin method delivers an approximation that is well-defined, independent of the choice of basis, techniques for enforcing bounds constraints depend rather delicately on this choice.
Determining the positivity of a multivariate polynomial is NP-hard~\cite{lasserre2007sum}, so here we restrict ourselves to some workable, if inexact approaches.

Let $\mathcal{P} = \{ \psi_i \}_{i=1}^{\dim\mathcal{V}_h}$ be a basis for $\mathcal{V}_h$, which we will typically take either to be the piecewise Lagrange ($\mathcal{P} = \mathcal{L}_k$) or Bernstein polynomials ($\mathcal{P} = \mathcal{B}_k$), with $k$ denoting the local polynomial degree. 
Then, we define the set
\begin{equation}
  \mathcal{J}^{\mathcal{P}, I}_h
  = \left\{ \sum_{i=1}^{\dim \mathcal{V}_h} c_i \psi_i : c_i \in I \right\},
\end{equation}
to comprise only members of $\mathcal{V}_h$ whose coefficients relative to the basis $\mathcal{P}$ satisfy the bounds constraints. Here, $I$ may be a closed interval of the form $[m, M]$, or a half-open interval of the form $(-\infty, M]$ or $[m, \infty)$.  This notation applies equally well to discontinuous polynomial spaces.

When $\mathcal{P}$ consists of piecewise Lagrange polynomials, $\mathcal{J}_h^{\mathcal{L}, I}$ contains polynomials that satisfy the bounds constraints at the interpolating nodes.
In this case, every member of $\mathcal{V}_h$ with range in $I$ lies in $\mathcal{J}_h^{\mathcal{L}, I}$, but so do many that do not -- the bounds are readily violated between nodes.  Figure~\ref{fig:edgecaselagrange} shows a simple example of a univariate polynomial on $[0, 1]$ for which this is the case.

In contrast, suppose that $\mathcal{P}$ is be the basis of piecewise Bernstein polynomials.  On the unit interval, the Bernstein basis of degree $n$
is given by
\begin{equation}
  b^n_i(x) = \binom{n}{i} x^i (1-x)^{n-i}, \ \ \ 0 \leq i \leq n.
\end{equation}
These polynomials give a nonnegative partition of unity forming a basis for polynomials of degree $n$.  They are readily mapped to any compact interval $[a, b]$, and given their geometric decomposition~\cite{arnold2009geometric}, they can be easily assembled across cells to form $C^0$ piecewise polynomials.  The cubic basis is given in Figure~\ref{fig:bern3}.

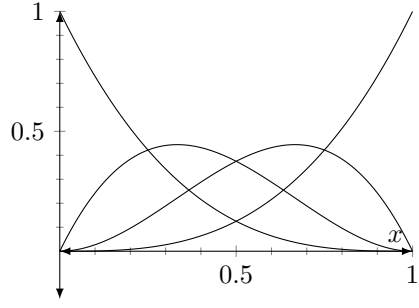
\begin{figure}
  \begin{center}
    \begin{tikzpicture}
      \begin{axis}[width=0.48\textwidth, grid style={line width=.1pt, draw=gray!10}, legend style={nodes={scale=0.7, transform shape}}, minor tick num=4, major grid style={line width=.2pt,draw=gray!50}, axis lines=middle,
            axis line style={latex-latex},
            legend pos=north east,
            ymin=-0.2,
            ymax=1,
            ytick={0.0, 0.5, 1.0},
            yticklabels={$0$, $0.5$, $1$},
            xtick={0,0.5,1}, 
            xticklabels={$0$, $0.5$, $1$},
            xlabel=$x$]

      \addplot[domain=0:1, samples=100, color=black] {(1-x)^3};
      \addplot[domain=0:1, samples=100, color=black] {3*x * (1-x)^2};
      \addplot[domain=0:1, samples=100, color=black] {3*(x)^2 * (1-x)};
      \addplot[domain=0:1, samples=100, color=black] {x^3};
      \end{axis}
  \end{tikzpicture}
  \end{center}
  \caption{Cubic Bernstein polynomials.}
  \label{fig:bern3}
\end{figure}

The Bernstein polynomials possess a convex hull property -- the graph of
\(
p(x) = \sum_{i=0}^n c_i b^n_i(x)
\)
for $x \in [0, 1]$ 
lies in the convex hull of its control net $\left\{\left( \tfrac{i}{n}, c_i \right)\right\}_{i=0}^n$~\cite{LaiSch07}.  So, if the coefficients $c_i \in I$, then the polynomial can only take values in $I$.
However, there exist polynomials that are uniformly positive on $[0, 1]$ but have negative coefficients in the Bernstein basis, as shown in
Figure~\ref{fig:edgecasebernstein}.
This situation is the opposite of the Lagrange basis -- every member of $\mathcal{J}_h^{\mathcal{B}, I}$ is bounds-constrained, but we only have a proper subset of constrained polynomials.

\begin{figure}
  \centering
  \begin{subfigure}[t]{.48\textwidth}
    \centering
    \begin{tikzpicture}
      \centering
        \begin{axis}[width=\textwidth, grid style={line width=.1pt, draw=gray!10}, legend style={nodes={scale=0.7, transform shape}}, minor tick num=4, major grid style={line width=.2pt,draw=gray!50}, axis lines=middle,
            axis line style={latex-latex},
            legend pos=north east,
            ymin=-0.2,
            ymax=1,
            ytick={0.0, 0.5, 1.0},
            yticklabels={$0$, $0.5$, $1$},
            xtick={0,0.5,1}, 
            xticklabels={$0$, $0.5$, $1$},
            xlabel=$x$]
          \addplot[domain=0:1, samples=100, color=black] {0.01*2*(x-0.5)*(x-1) - 0.01*4*x*(x-1) + 2*x*(x-0.5)};
          \addlegendentry{$0.01\ell_0(x) + 0.01\ell_1(x) + 1\ell_2(x)$}
        \end{axis}
    \end{tikzpicture}
    \subcaption{A polynomial with positive coefficients in the Lagrange basis that is not uniformly positive.}
    \label{fig:edgecaselagrange}
  \end{subfigure}\hspace{0.04\textwidth}%
  \begin{subfigure}[t]{.48\textwidth}
    \centering
    \begin{tikzpicture}
      \centering
        \begin{axis}[width=\textwidth, grid style={line width=.1pt, draw=gray!10}, legend style={nodes={scale=0.7, transform shape}}, minor tick num=4, major grid style={line width=.2pt,draw=gray!50}, axis lines=middle,
            axis line style={latex-latex},
            legend pos=north east,
            ymin=-0.2,
            ymax=1,
            ytick={0.0, 0.5, 1.0},
            yticklabels={$0$, $0.5$, $1$},
            xtick={0,0.5,1}, 
            xticklabels={$0$, $0.5$, $1$},
            xlabel=$x$]
          \addplot[domain=0:1, samples=100, color=black] {(1-x)^2 -1.8*x*(1-x)+x^2};
          \addlegendentry{$1b_0(x) -0.9b_1(x) + 1b_2(x)$}
        \end{axis}
    \end{tikzpicture}
    \subcaption{A uniformly positive polynomial with a negative coefficient in the Bernstein basis.}
    \label{fig:edgecasebernstein}
  \end{subfigure}
  \caption{(Left): An admissible polynomial in the nodally-constrained Lagrange basis. (Right): An inadmissible polynomial in the uniformly-constrained Bernstein basis.}
  \label{fig:edgecasepolys}
\end{figure}
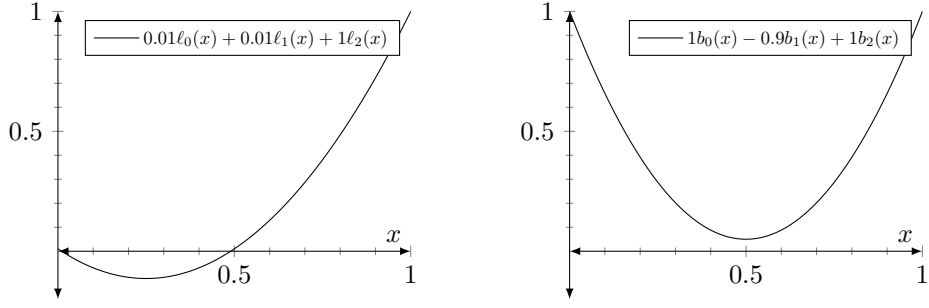 

The Bernstein polynomials are easily extended to the $d$-simplex, and analogous results hold. In particular, let $T\subset \mathbb{R}^d$ nondegenerate simplex. Let $\mathrm{b} = (b_0, \dots b_d)$ be the vector of barycentric coordinates of $T$. Then, for any multiindex $\alpha$ with $|\alpha| = n$, the associated Bernstein basis polynomial of degree $n$
is given by 
\begin{equation}
  B^\alpha = \frac{n!}{\alpha!} \mathbf{b}^{\alpha}.
\end{equation}
The Bernstein basis of degree $n$ is given by the collection
\begin{equation}
  \left\{B^\alpha\right\}_{|\alpha| = n}.
\end{equation}
The multivariate Bernstein basis is also geometrically decomposed, making the assembly of a $C^0$ piecewise finite element basis straightforward. 
In the event we are discretizing a subset of $L^2(\Omega)$, say, with a discontinuous Galerkin method, the interelement continuity need not be imposed for the bounds constraints to hold.

We define $\mathcal{J}^{\mathcal{P}, [m,M]}_h$ to be those members of $\mathcal{V}_h$ whose degrees of freedom in the chosen basis (either Lagrange or Bernstein) satisfy the bounds constraints.  
With the Lagrange basis, we obtain all members of $\mathcal{V}_h$ that satisfy the bounds constraints uniformly, but have an outside approximation since members may violate those bounds between the Lagrange nodes.
With the Bernstein basis, $\mathcal{J}_h^{\mathcal{B}, [m,M]}$ contains only functions uniformly satisfying the bounds constraints, but not all of them.

This construction, in either the Lagrange or Bernstein basis, allows us to define (discretely) feasible sets.   
For example, discretizing the heat equation with nonnegative values gives $\mathcal{K}_h = \mathcal{J}_h^{\mathcal{P}, [0, \infty)}$, 
  while constraining the quantity parameter $q$ in the advection equation gives $\mathcal{K}_h = \mathcal{J}_h^{\mathcal{P}, [m, M]}$.

While we focus here on constraining the degrees of freedom in the Lagrange and Bernstein bases, there are additional approaches to ensuring uniformly-constrained polynomial approximations. In~\cite{dzanic2025method}, bounding boxes are 
constructed once by solving a constrained optimization problem which may then be used to produce polynomials in the chosen basis which are uniformly-constrained. 
Introducing such constraints into the methods presented here would allow for the construction of uniformly-constrained methods with other bases, but would require the use of solvers which support more general inequality constraints. We restrict 
our attention to the uniformly-constrained Bernstein basis and the nodally-constrained Lagrange basis.

\section{Method Development}\label{sec:methodDev}

We now formulate two kinds of new discretizations.
First, \emph{projective} methods compose a standard time discretization with a nonlinear projection into the feasible set.
Also, certain methods admit a monolithic formulation, discretized directly with a variational inequality.

Before proceeding, we summarize the main structure of Runge--Kutta and multistep methods.
We partition the time interval $[0, T]$ into intervals
$(t^n, t^{n+1})$ with $t^0 = 0$ and $t^N = T$.
As a notational convenience, we will assume a uniform time step size of $t^{n+1} - t^n = k$.
There is no practical limitation of Runge-Kutta methods to uniform step sizes, but care must 
be taken when working with multistep methods.
The bounds-constrained methods presented here do not require additional restrictions.
Throughout, $y^{n}$ denotes an approximation to the exact solution at time $t^{n}$.

Typically, an $s$-stage Runge-Kutta method is encoded by a Butcher tableau
\begin{equation*}
  \begin{array}{c|c}
    \mathbf{c} & A \\ \hline
    & \mathbf{b}
  \end{array},
\end{equation*}
where $\mathbf{b}, \mathbf{c} \in \mathbb{R}^s$ and $A \in \mathbb{R}^{s \times s}$.

Applying a Runge-Kutta scheme to~\eqref{eq:model_problem} requires constructing the new solution by first solving a variational problem for stage variables.
We define $\mathcal{V}_h^s$ as the $s$-way Cartesian product $\mathcal{V}_h^s = \prod_{i=1}^s \mathcal{V}_h = \mathcal{V}_h \times \mathcal{V}_h \times \dots \times \mathcal{V}_h$.

Given an approximation $y^n$ to $y$ at time $t^n$, the Runge-Kutta scheme seeks $Y \in \mathcal{V}_h^s$ such that
\begin{equation}
  \label{eq:stagedef}
  \left(Y_i, v_i\right)
  = \left( y^n, v_i \right) + k \sum_{j=1}^s A_{ij}
  \langle F_{h,t^n + \mathbf{c}_j k}(Y_j), v_i \rangle
\end{equation}
for all $v_i \in \mathcal{V}_h, 1 \leq i \leq s$.
The stage values $Y_j$ approximate the solution $y$ at time $t^n +\mathbf{c}_j k$.
We may write this abstractly as a variational problem seeking $Y \in \mathcal{V}^s_h$ such that
\begin{equation}
  \label{eq:abstractstage}
  \langle \mathcal{F}(Y), V\rangle = 0, \qquad \forall V \in \mathcal{V}_h^s.
\end{equation}

When the Butcher matrix $A$ is strictly lower triangular, the method is explicit, and the stages $Y_i$ may be computed sequentially with a simple mass matrix to invert each time.

Then, $y^{n+1} \in \mathcal{V}_h$ arises from the much simpler variational problem
\begin{equation}
  \label{eq:updatestage}
  \left( y^{n+1}, v \right)  = \left( y^n, v \right)
  + k \sum_{j=1}^s b_j \langle F_{h, t_n + \mathbf{c}_j k}(Y_j), v \rangle \qquad \forall v\in \mathcal{V}_h.
\end{equation}
This must be posed variationally since $F_{h,t}$ maps $\mathcal{V}_h$ into its dual.
For stiffly-accurate Runge-Kutta methods, this extra variational problem is not needed, as we have just $y^{n+1} = Y_s$, the final stage value.

In certain simple cases, Runge--Kutta schemes can be posed directly for $y^{n+1}$.
These include single-stage methods (theta methods including forward and backward Euler and the implicit midpoint rule) as well as two-stage methods with an explicit first stage such as Crank-Nicolson.  In this case, we write a variational problem for $y^{n+1}$ in the abstract form
\begin{equation}
  \langle \mathcal{F}(y_{n+1}; y_n, k), v \rangle = 0\qquad \forall v\in\mathcal{V}_h.
\end{equation}

We also consider general multistep methods.
Given a sequence of approximations $\{ y^{n-s+j} \}_{j=1}^s$,
to the solution of~\eqref{eq:model_problem} at the discrete times $t^{n-s+j}$,
an $s$-step multistep method gives an
approximation $y^{n+1}$ at time $t^{n+1}$
via the solution of the variational problem seeking $y^{n + 1}\in \mathcal{V}_h$ such that
\begin{equation}\label{eq:BDF_general}
  \sum_{j = 0}^s a_j (y^{n-s+j+1}, v) =
  k\sum_{j = 0}^s b_j \langle F_{h,t^{n-s+j+1}}(y^{n-s+j+1}), v\rangle, \qquad \forall v\in \mathcal{V}_h.
\end{equation}

Particular choices of the $a$ and $b$ coefficients lead to well-known multistep schemes, such as the backward difference (BDF) schemes and Adams-type methods~\cite{bashforth1883attempt, moulton1926new}.  For example, the 2-step BDF method produces the approximation at time $t^{n + 1}$ by solving the variational problem seeking $y^{n + 1}\in \mathcal{V}_h$ such that
\begin{equation}
  (y^{n + 1}, v) - \frac{4}{3}(y^n, v) + \frac{1}{3}(y^{n - 1}, v) = \frac{2}{3}k\langle F_{h,t^{n + 1}}(y^{n + 1}), v\rangle\qquad\forall v\in \mathcal{V}_h.
\end{equation}
Multistep methods may also be written as an abstract variational problem seeking $y^{n + 1}\in \mathcal{V}_h$ such that 
\begin{equation}
  \langle \mathcal{F}(y^{n+1}; Y^n), v\rangle = 0\qquad \forall v\in \mathcal{V}_h,
\end{equation}
where $Y^n = [y^n, \dots, y^{n - s + 1}]$ contains the approximations at the previous time steps. 

Our presentation of projective methods will benefit from a simple abstract notation unifying Runge--Kutta and multistep schemes.
Here, we introduce a mapping
$\Phi:\mathcal{V}_h^s \rightarrow \mathcal{V}$, where $s=1$ for Runge-Kutta schemes and $s>1$ for multistep.  Putting $Y^n = [y^n, y^{n-1}, \dots, y^{n - s + 1}] \in \mathcal{V}_h^s$, any of these methods takes the form
\begin{equation}
  y^{n+1} = \Phi(Y^n),
\end{equation}
where $\Phi$ represents the computation required to compute $y^{n+1}$ from one or more previous values of solution by a particular scheme.


\subsection{Projective Methods}


The projection of a function onto a closed and convex subset requires the solution of a constrained variational inequality, or, equivalently, a constrained minimization problem. Recall the following definition (see, e.g., \cite{kinderlehrer1980}):
\begin{definition}
For a Hilbert space $\mathcal{H}$ with inner product $(\cdot, \cdot)$, and a closed and convex set $\mathcal{K}\subset \mathcal{H}$, let $\mathrm{P}_\mathcal{K}(u)$ denote 
the unique solution $u_c\in \mathcal{K}$ to
\begin{equation}
(u_c - u, v - u_c) \geq 0\quad \forall v\in \mathcal{K}.
\end{equation}
\end{definition}
This is equivalent to $u_c$ being the unique minimizer of $\inf_{v\in\mathcal{K}} ||u - v||$.

Given a method $\Phi$, the projective method is obtained by composition with the nonlinear projection. Let $\Phi_{\mathcal{K}_h}: \mathcal{V}_h^s\rightarrow \mathcal{K}_h$ be given by
\begin{equation}
  \Phi_{\mathcal{K}_h}(Y) = \left(\mathrm{P}_{\mathcal{K}_h} \circ \Phi\right)(Y)
\end{equation}
for initial data $Y\in \mathcal{V}_h^s$. The projective method $\Phi_{\mathcal{K}_h}$ can then be used in place of $\Phi$ to compute a sequence of approximations $\{y^{n}\}_{n = 1}^N\subset \mathcal{K}_h$.

This approach can be composed with any time-stepping method to produce
a discretely feasible solution at each time level $t^n$.
Moreover, it is non-invasive so that any fast solvers available for the unconstrained method such as those in~\cite{abu2022monolithic,clines2022efficient,mmg,masud2021new,southworth2022fast2,vanlent2005} are directly applicable.

\subsection{Monolithic Methods}\label{sec:monolithicMethods}

Projection-type methods provide an effective and straightforward method for the enforcement of bounds constraints.
This turns out to be quite sufficient in many cases, as we will see in our numerical results.
However, solvers for nonlinear problems may require that intermediate states be feasible, which projective methods cannot do.

Suppose that our time-stepper $\Phi$ amounts to solving a variational equation for $y^{n+1} \in \mathcal{V}_h$ of the form
\begin{equation}
  \langle \mathcal{F}(y^{n+1}; Y^n), v \rangle = 0,
\end{equation}
for $v \in \mathcal{V}_h$, 
where $y^{n+1}$ is the new solution and $Y^n$ is the vector of prior solutions.  This formulation captures single-stage Runge--Kutta methods and multistep methods.

Here, as in~\cite{barrenechea2025nodally,kirby2024variational}, we simply replace the variational equation with the variational inequality
seeking $y_c^{n+1} \in \mathcal{K}_h$ such that
\begin{equation}
  \langle \mathcal{F}(y_c^{n+1}; Y^n), v - y_c^{n+1} \rangle \geq 0
\end{equation}
for all $v \in \mathcal{K}_h$.

For example, the implicit midpoint method may be written as a variational inequality seeking $y_c^{n + 1}\in \mathcal{K}_h$ such that 
\begin{equation}
  \left(
  (y_c^{n + 1}, v - y_c^{n + 1}\right) \geq (y^n, v - y_c^{n + 1}) + k \left\langle F_{t^{n + 1/2}}\left(\frac{y_c^{n + 1} + y^{n}}{2}\right), v - y_c^{n+1}\right\rangle
\end{equation}
for all $ v\in \mathcal{K}_h.$  Similarly, the BDF(2) method can be reformulated to seek $y_c^{n + 1}\in \mathcal{K}_h$ such that  
\begin{equation}
  \left( y_c^{n+1} - \tfrac{4}{3} y^n + \tfrac{1}{3} y^{n-1} , v - y_c^{n+1} \right)
  \geq \frac{2}{3}k\langle F_{t^{n + 1}}(y_c^{n + 1}), v - y_c^{n + 1}\rangle
\end{equation}
for all $v \in \mathcal{K}_h$.

There are fundamental differences between this approach and the projective methods above.
We must solve a more complex discrete variational inequality, and solvers and preconditioners for the unconstrained problem may not apply directly.
Our numerical results indicate that the two approaches yield similar numerical solutions for linear problems.
Although the projective methods are significantly faster than monolithic ones, nonlinear problems may benefit from greater robustness of monolithic schemes.

\subsection{Preservation of Other Quantities}
PDEs may have other conserved quantities like mass or energy.  Even if the underlying spatial discretization preserves these, the postprocessing or monolithic modifications of the schemes we have introduced may not.
However, we can pose our variational inequalities over
any closed and convex approximating set.

For instance, 
in the advection problem~\eqref{eq:advec_eq}, the solution satisfies bounds constraints in addition to the linear invariant
\begin{equation}
  \int_\Omega u \mathrm{dx} = \int_\Omega u_0 \mathrm{dx}.
\end{equation}
Preserving only the bounds constraints, the discrete feasible set
\begin{equation}
  \mathcal{K}_h = \mathcal{J}_h^{\mathcal{P}, [m, M]}
\end{equation}
may be used. If the mass constraint is also desired, then the feasible set becomes
\begin{equation}
  \mathcal{K}_h =  \mathcal{J}_h^{\mathcal{P}, [m, M]}\cap \left\{u \in L^2(\Omega) : \int_\Omega u \mathrm{dx} = \int_\Omega u_0 \mathrm{dx}\right\}.
\end{equation}
However, this may complicate the implementation, and the optimization solver must support the additional constraints.

\section{Some notes on analysis}\label{sec:approxTheory}

The following analysis relies heavily on the nonexpansivity of projections:

\begin{theorem}[Kinderlehrer\cite{kinderlehrer1980}]\label{thm:nonexpansivity}
  Let $\mathcal{H}$ be a Hilbert space and $\mathcal{K}$ a closed and convex subset. Then, for any $u, v\in \mathcal{H}$,
  \begin{equation*}
    ||\mathrm{P}_\mathcal{K} u - \mathrm{P}_\mathcal{K}v ||\leq ||u - v||.
  \end{equation*}
\end{theorem}

If the solution to the evolution equation~\eqref{eq:model_problem} is (discretely) admissible, $y(t)\in \mathcal{K}_h$ for all times $t>0$, a very general result holds:
\begin{proposition}\label{prop:generalProjError}
  Let $\Phi$ denote a time-stepping scheme. Let $Y^n\in\mathcal{V}_h^s$ be a sequence of approximations. Suppose $y(t)\in \mathcal{K}_h$ for all $t > 0$. Then, for all $n > 0$
  \begin{equation}
    \left|\left|y(t^{n + 1}) - \Phi_{\mathcal{K}_h}(Y^{n})\right|\right| \leq \left|\left|y(t^{n + 1}) - \Phi(Y^{n}) \right|\right|
  \end{equation}
\end{proposition}
\begin{proof}
  Suppose $y(t)\in \mathcal{K}_h$ for all $t > 0$. Then, by the definition of $\Phi_{\mathcal{K}_h}$ and Theorem~\ref{thm:nonexpansivity},
  \begin{multline*}
    \left|\left|y(t^{n + 1}) - \Phi_{\mathcal{K}_h}(Y^n)\right|\right| = \left|\left|y(t^{n + 1}) - \mathrm{P}_{\mathcal{K}_h}(\Phi(Y^n))\right|\right|\\ = \left|\left|\mathrm{P}_{\mathcal{K}_h}(y(t^{n + 1})) - \mathrm{P}_{\mathcal{K}_h}(\Phi(Y^{n}))\right|\right| \leq \left|\left| y(t^{n + 1}) - \Phi(Y^{n})\right|\right|
  \end{multline*}
\end{proof}

This result only applies in a limited case -- bounds-preserving semidiscrete systems with discretely feasible solution -- it allows for an elegant error result. After each time step, the error accumulated by the bounds constrained scheme 
is less than that accumulated by the unconstrained method.
Established estimates for the unconstrained method $\Phi$, therefore, apply immediately.

When the underlying PDE is bounds-constrained but the spatial discretization is not, another level of analysis is required.
When we do not have that $y(t)\in \mathcal{K}_h$, the following may be said:
\begin{proposition}\label{prop:generalProjDiscError}
    Let $\Phi$ denote a time-stepping scheme. Let $Y^n\in\mathcal{V}_h^s$ be a sequence of approximations. Then, for all $n > 0$,
    \begin{equation*}
        \left|\left|y(t^{n + 1}) - \Phi_{\mathcal{K}_h}(Y^{n})\right|\right|\leq \left|\left|y(t^{n + 1}) - \mathrm{P}_{\mathcal{K}_h}(y(t^{n + 1})) \right|\right| + \left|\left|y(t^{n + 1}) - \Phi(Y^{n}) \right|\right|.
    \end{equation*}
\end{proposition}
\begin{proof}
    Suppose $\Phi$ is a time-stepping method. Then, by the triangle inequality, the definition of $\Phi_{\mathcal{K}_h}$, and Theorem~\ref{thm:nonexpansivity},
    \begin{align*}
        \left|\left|y(t^{n + 1}) - \Phi_{\mathcal{K}_h}(Y^{n})\right|\right| &\leq \left|\left|y(t^{n + 1}) - \mathrm{P}_{\mathcal{K}_h}(y(t^{n + 1}))\right|\right| + \left|\left|\mathrm{P}_{\mathcal{K}_h}(y(t^{n + 1})) - \Phi_{\mathcal{K}_h}(Y^{n})\right|\right|\\
                                                                                            &\leq \left|\left|y(t^{n + 1}) - \mathrm{P}_{\mathcal{K}_h}(y(t^{n + 1}))\right|\right|\\ &\qquad\qquad\qquad\qquad + \left|\left|\mathrm{P}_{\mathcal{K}_h}(y(t^{n + 1})) - \mathrm{P}_{\mathcal{K}_h}(\Phi(Y^{n}))\right|\right|\\
                                                                                            &\leq \left|\left|y(t^{n + 1}) - \mathrm{P}_{\mathcal{K}_h}(y(t^{n + 1}))\right|\right| + \left|\left|y(t^{n + 1}) - \Phi(Y^{n})\right|\right|.
    \end{align*}
\end{proof}


The nonexpansivity of projections can also be used to show that if $0\in \mathcal{K}_h$ and the unconstrained method $\Phi$ is stable, then the constrained method $\Phi_{\mathcal{K}_h}$ is stable in the same sense. 
A similar result holds for the monolithic backward Euler method. Further stability and error analysis is left as a topic of future study.

\section{Numerical Results}\label{sec:examples}

We have formulated several broad classes of methods, and we give numerical results for several important instances.
We use the RadauIIA and Gauss-Legendre families of implicit Runge-Kutta methods and the BDF family of multistep methods.  For a labeling scheme, we use the following: the spatial finite element will be denoted by $\mathcal{P}_k$, where 
$\mathcal{P}$ will take on either $\mathcal{L}$ in the event the Lagrange basis is used, or $\mathcal{B}$ for the Bernstein parameterization, and $k$ denotes the polynomial degree of the finite element space.
The time stepper will be 
denoted by $\mathrm{RIIA(s)}$, $\mathrm{GL(s)}$, or $\mathrm{BDF(s)}$, with $s$ denoting the number of stages or steps.
Finally, we append $\mathrm{P}$ to the scheme label to indicate the 
bounds-preserving projective modification and $\mathrm{VI}$ to indicate the monolithic scheme.
For example, we denote the method using quadratic Bernstein elements in space,
2-stage RadauIIA in time, and bounds-preserving projection by
$\mathcal{B}_2$-$\mathrm{RIIA}(2)$-$\mathrm{P}$.

All numerical experiments were completed using the Firedrake Project~\cite{Rathgeber:2016}, a Python package designed for the efficient implementation of finite element methods. 
The Irksome project~\cite{farrell2021irksome,kirby2024extending} sits on top of Firedrake to efficiently apply Runge--Kutta methods to finite element discretizations, 
and was used extensively throughout. Firedrake provides high-level access to PETSc~\cite{petsc-user-ref,petsc-efficient} through petsc4py~\cite{dalcin2011}. When using Lagrange elements, the degrees of freedom are chosen in accordance with~\cite{isaac2020recursive} rather than using equispaced nodes.

All monolithic variational inequalities were solved using the Newton-type reduced space VI solver {\ttfamily vinewtonrsls}, as implemented in PETSc.  Constrained $L^2$ projection 
we carried out using {\ttfamily vinewtonrsls}, the conjugate-gradient method for the resulting linear systems, and preconditioned with SOR.
The $L^2$ bounds and mass-preserving projection was computed using the augmented Lagrangian method {\ttfamily ALMM} as implemented in PETSc's Toolkit for Advanced Optimization~\cite{hestenes1969multiplier,powell1969method}. For all linear and nonlinear problems, we set the absolute tolerance to $10^{-8}$.
We describe the particular linear solvers used for each problem below.
All numerical examples are carried out on a single core of a laptop with 16GB of memory.  All timing results are averaged over five identical computations.

\subsection{The heat equation}
Here, we consider the heat equation.
Before studying the convergence rates and timing analysis of the proposed methods, we demonstrate possible failure modes in our methods
on a simple example.
We make two concrete distinctions in this first example: the difference between post-processing and monolithic methods, and that between the Lagrange basis and the Bernstein basis.

We first choose the initial condition, boundary condition, and forcing function so that the exact solution is 
\begin{equation}
  u = C(x, y) \left(\frac{1}{2} + \frac{1}{2}\textrm{tanh}(75t - 6)\right),
\end{equation}
where
\begin{equation}\label{eq:heat_C}
  C(x, y) = \left(\frac{1}{2} - \frac{1}{2}\textrm{tanh}\left(\frac{0.15 - \sqrt{\left(x - \frac{1}{2}\right)^2 + \left(y - \frac{1}{2}\right)^2}}{0.015}\right)\right).
\end{equation}
We utilize a uniform triangular mesh, $\mathcal{T}_h$, consisting of eight triangles. The computational mesh and associated degrees of freedom are shown in Figure~\ref{fig:heat_mesh}. We use continuous finite elements of degree $p=6$, and advance the system using the backward Euler method, constrained to enforce nonnegativity of the 
numerical solution using both a post-processing technique and a monolithic one. In both cases, the feasible set is given by $\mathcal{K}_h = \mathcal{J}_h^{\mathcal{B}_6, [0, \infty)}$. We take a single step of size $k = 1 / 4$.

\begin{figure}[ht]
  \centering
  \begin{subfigure}[t]{0.48\textwidth}
    \centering
    \includegraphics[width=\textwidth]{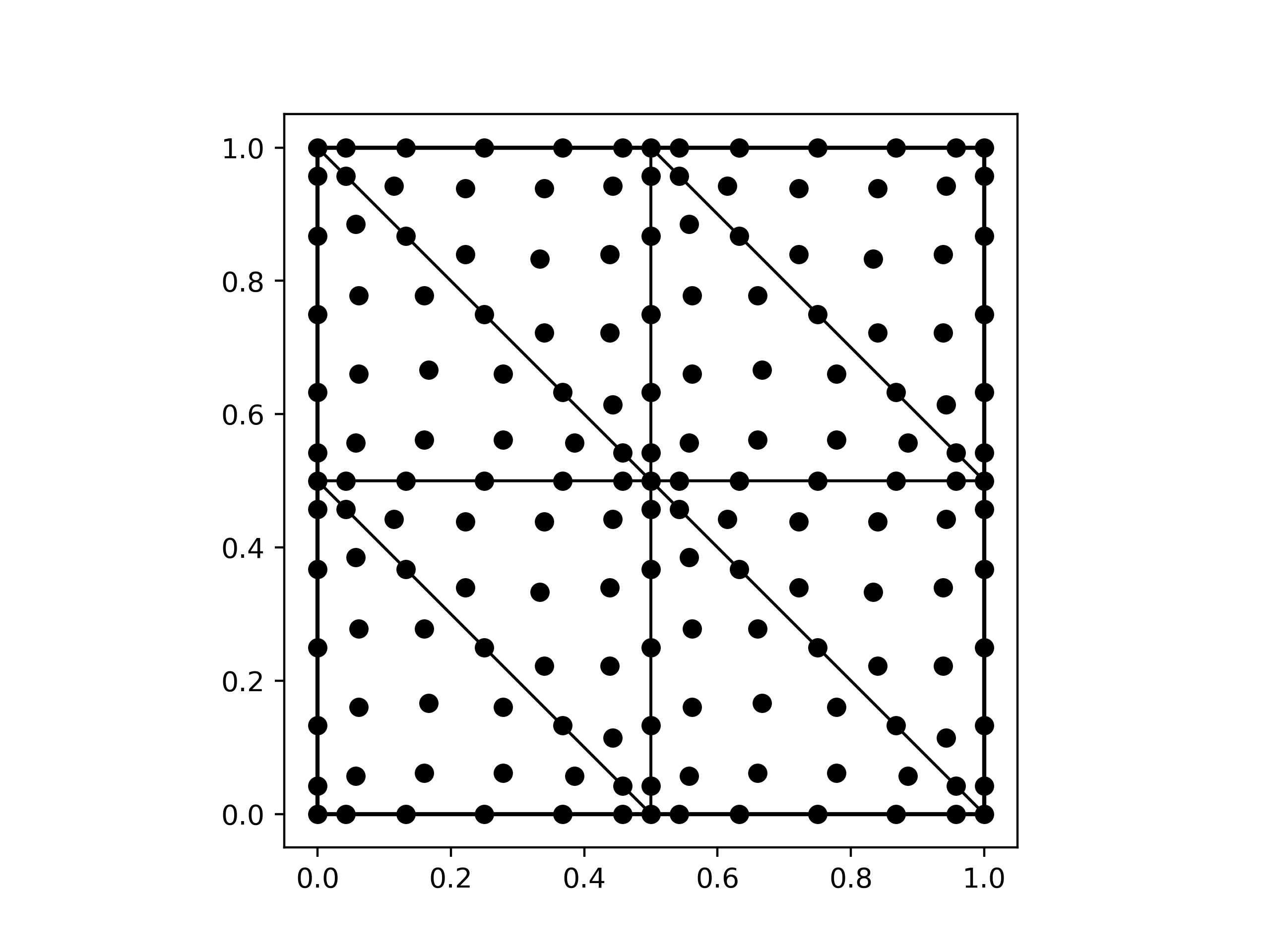}
  \end{subfigure}\hspace{0.04\textwidth}%
  \begin{subfigure}[t]{0.48\textwidth}
    \centering
    \includegraphics[width=\textwidth]{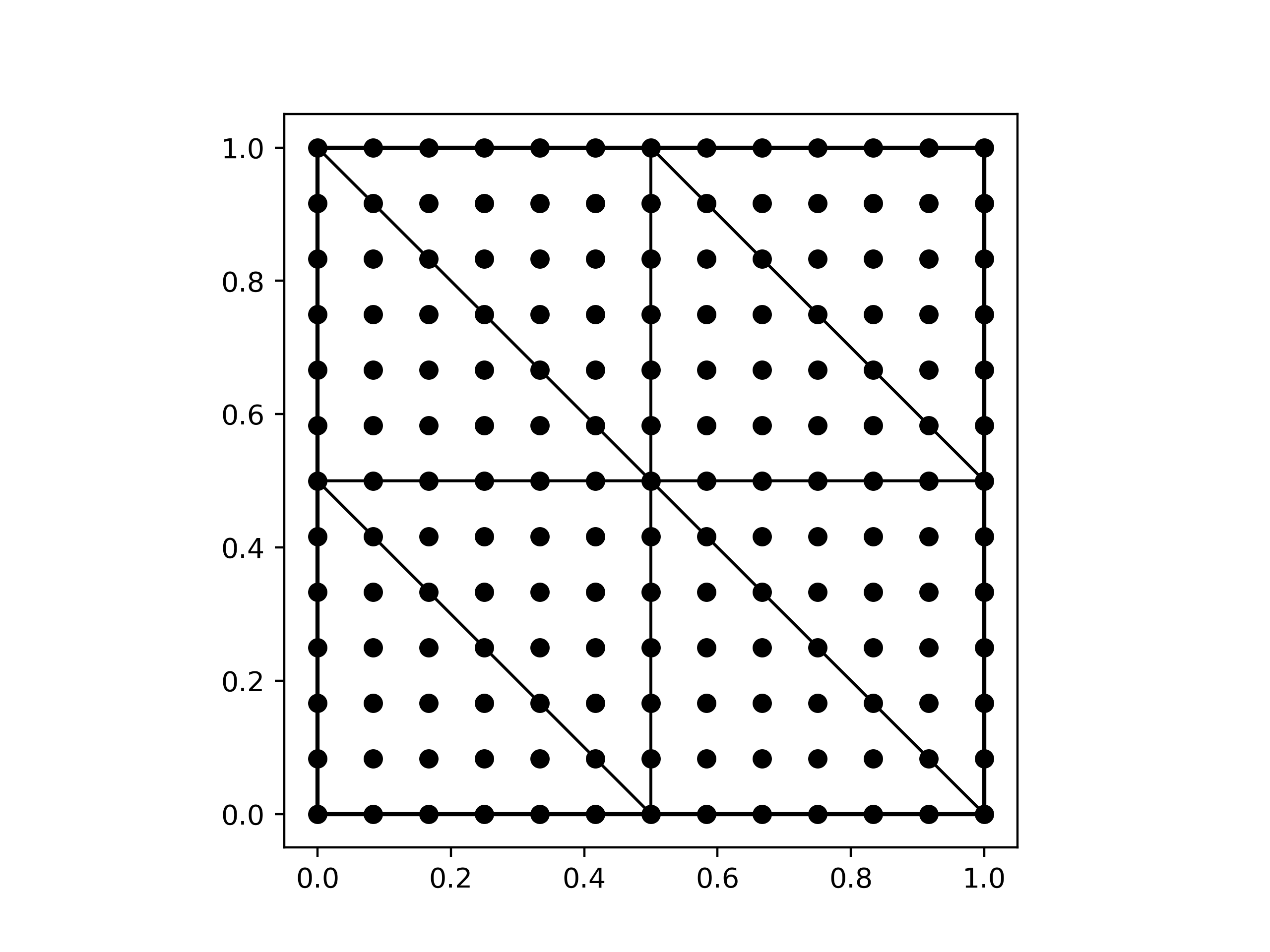}
  \end{subfigure}

  \caption{Computational mesh and the location of the degrees of freedom using the Lagrange basis (left) and the Bernstein basis (right).}
  \label{fig:heat_mesh}
\end{figure}

In Figure~\ref{fig:procedure_heat_post}, snapshots of the post-processing procedure are shown for $\mathcal{B}_6$-$\textrm{RIIA}(1)$-$\textrm{P}$.
Bounds violations are colored in red, and the 
plotted circles represent the degrees of freedom of the spatial function space. Before the projection, 
there are significant undershoots in the approximate solution, but the projection, removes these. In contrast, Figure~\ref{fig:procedure_heat_mono} shows snapshots of the approximation procedure when using the monolithic approach $\mathcal{B}_6$-$\textrm{RIIA}(1)$-$\textrm{VI}$. The initial conditions 
are identical, but the constrained approximation at time $t = 0.25$ is obtained without ever computing an intermediate approximation. The constraints are enforced during the solution of the discretized system, rather 
than by a bounds-preserving correction after the time step.

\begin{figure}[ht]
\vspace{1cm}
  \centering
  \begin{subfigure}[t]{0.3\textwidth}
    \centering
  \tikz[remember picture, baseline]{\node(1L){\includegraphics[width=0.9\textwidth]{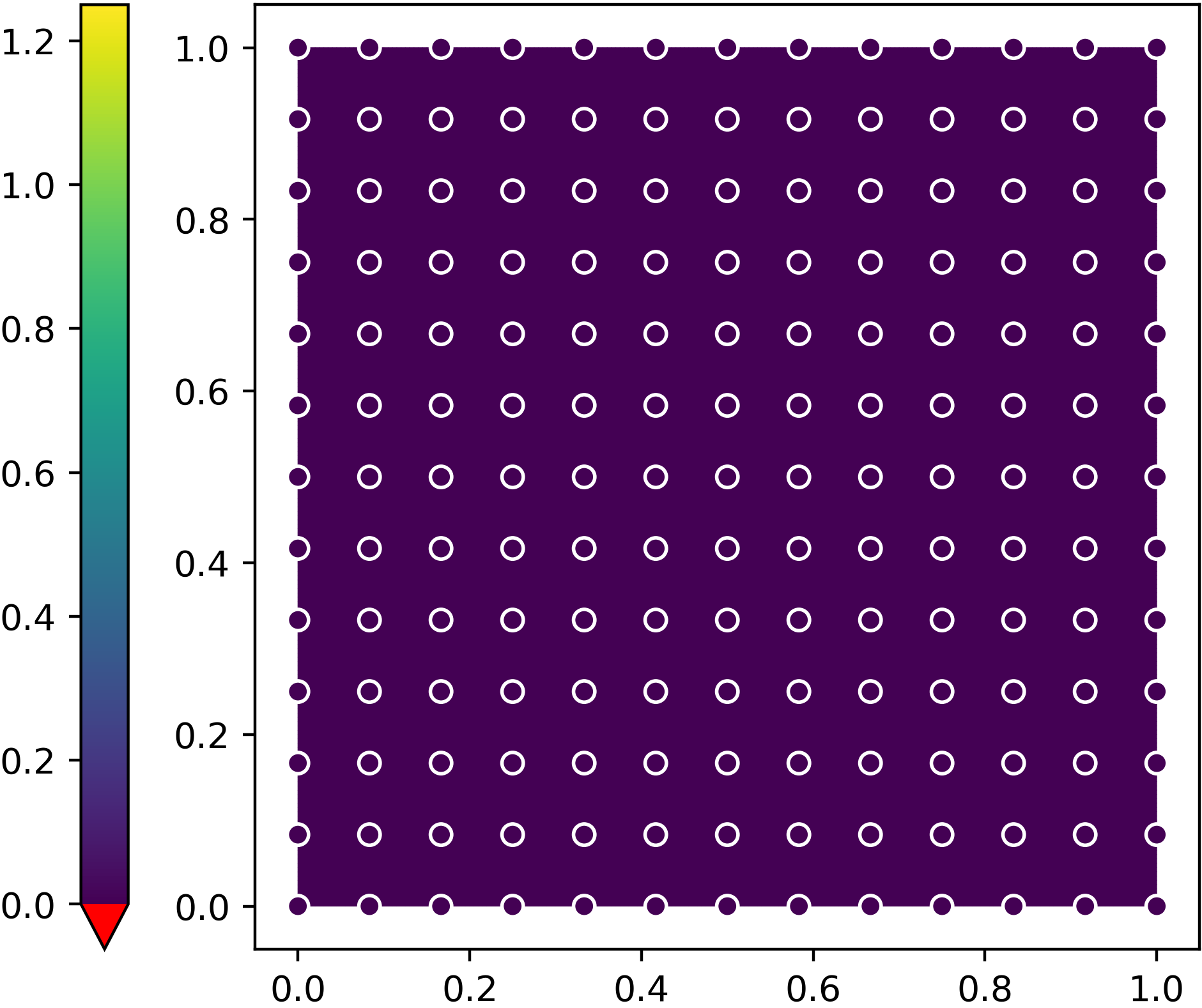}}}
  \subcaption{$t = 0.0$}
  \end{subfigure}\hspace{0.033\textwidth}%
  \begin{subfigure}[t]{0.3\textwidth}
    \centering
  \tikz[remember picture, baseline]{\node(1M){\includegraphics[width=0.8\textwidth]{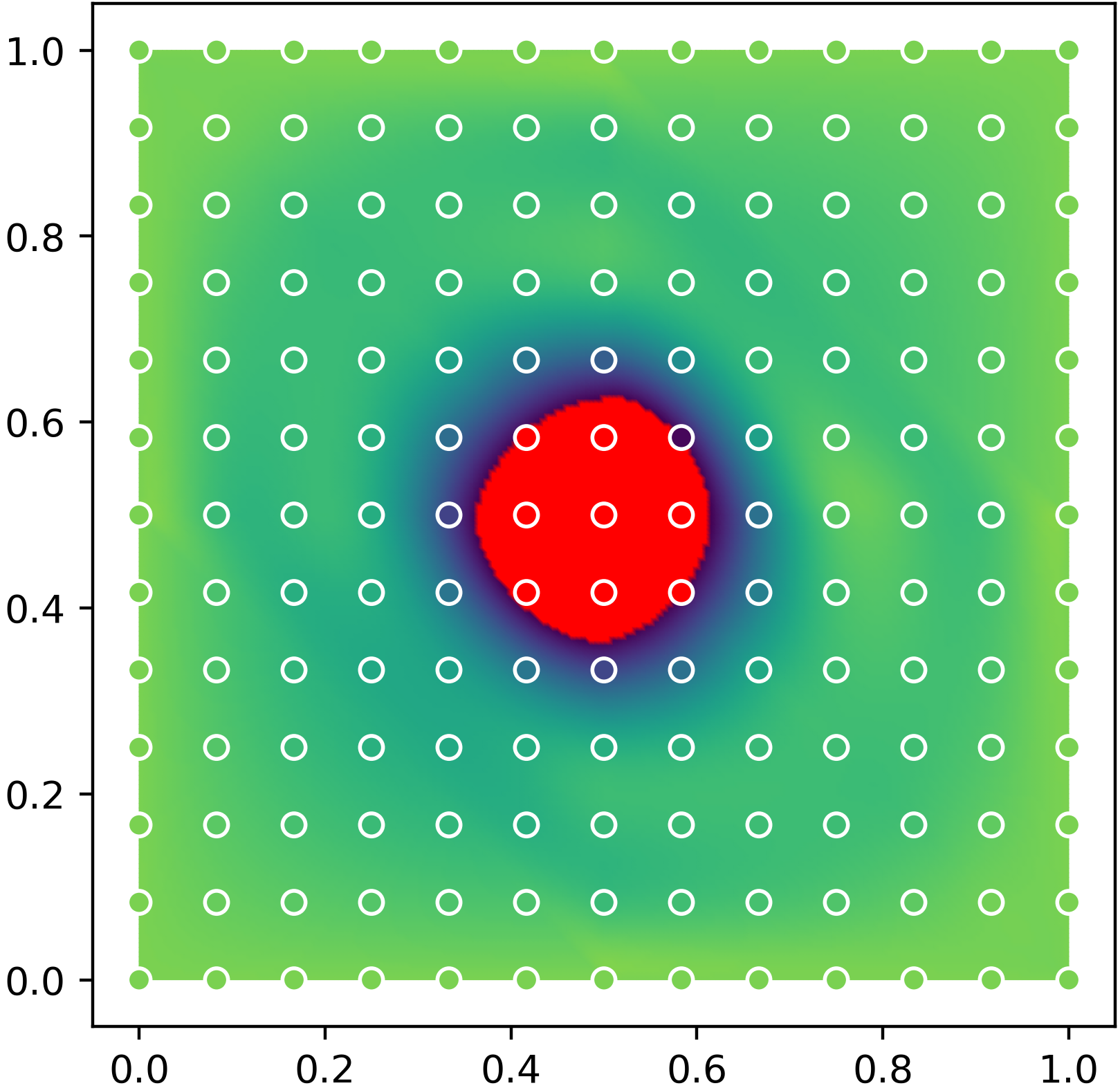}}}
  \subcaption{$t = 0.25$, pre-projection}
  \end{subfigure}\hspace{0.033\textwidth}%
  \begin{subfigure}[t]{0.3\textwidth}
    \centering
  \tikz[remember picture, baseline]{\node(1R){\includegraphics[width=0.8\textwidth]{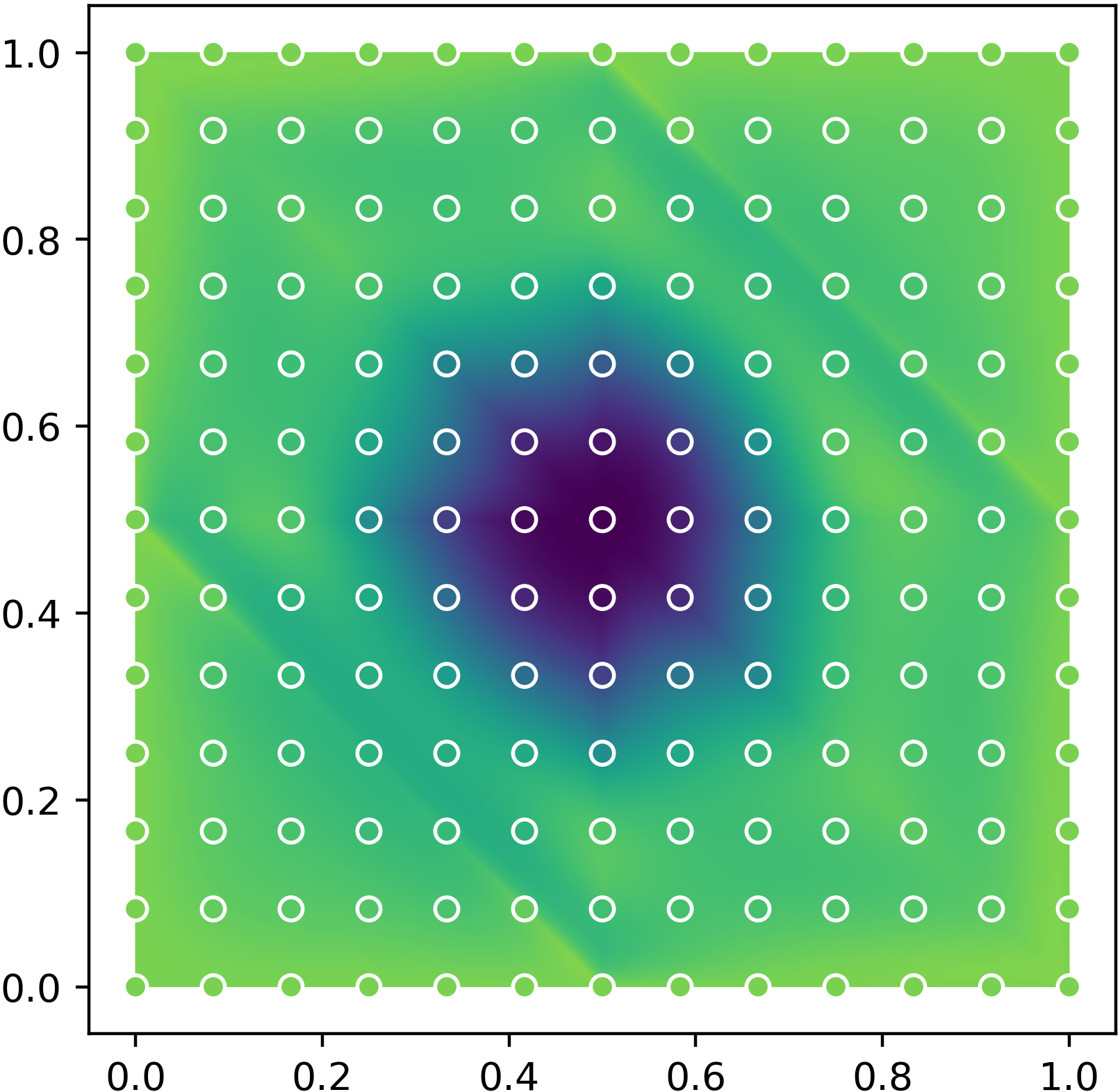}}}
  \subcaption{$t = 0.25$, projected}
  \end{subfigure}

  \caption{Snapshots of the post-processing procedure for integrating the heat equation using the $\mathcal{B}_6$-$\textrm{RIIA}(1)$-$\textrm{P}$ method.}
  \label{fig:procedure_heat_post}
  \tikz[overlay,remember picture]{
      \draw[-latex,very thick] (1L.north) .. controls +(up:1cm) and +(up:1cm) .. (1M.north)
          node[midway, above, align=center]{$\mathcal{B}_6$-$\textrm{RIIA}(1)$};
      \draw[-latex,very thick] (1M.north) .. controls +(up:1cm) and +(up:1cm) .. (1R.north)
          node[midway, above, align=center]{$\mathcal{B}_6$-$\textrm{P}_{L^2}$};
  }
\end{figure}
\begin{figure}[ht]
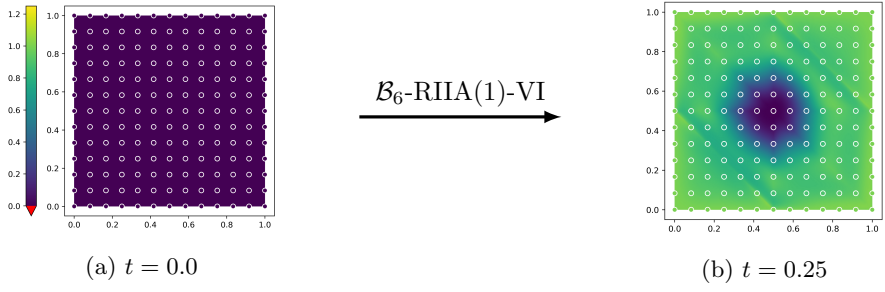

  \centering
  \begin{subfigure}[t]{0.3\textwidth}
    \centering
  \tikz[remember picture, baseline]{\node(2L){\includegraphics[width=0.9\textwidth]{B6RadauIIA1L2projTrue0.0.png}}}
  \subcaption{$t = 0.0$}
  \end{subfigure}\hspace{0.033\textwidth}%
  \hspace{0.3\textwidth}%
  \begin{subfigure}[t]{0.3\textwidth}
    \centering
  \tikz[remember picture, baseline]{\node(2R){\includegraphics[width=0.8\textwidth]{B6RadauIIA1L2projTrue_postproj_0.25.png}}}
  \subcaption{$t = 0.25$}
  \end{subfigure}

  \caption{Snapshots of the monolithic procedure for the integrating the heat equation using the $\mathcal{B}_6$-$\textrm{RIIA}(1)$-$\textrm{VI}$ method.}
  \label{fig:procedure_heat_mono}
  \tikz[overlay,remember picture]{
        \draw[-latex,very thick, shorten >=1cm, shorten <=1cm] (2L.east) -- (2R.west) 
          node[midway, above, align=center]{$\mathcal{B}_6$-$\textrm{RIIA}(1)$-$\textrm{VI}$};
  }
\end{figure}

If, instead, the spatial elements are parameterized using the Lagrange basis, $\mathcal{K}_h = \mathcal{J}_h^{\mathcal{L}_6, [0, \infty)}$, we observe slightly different behavior. In this case, the positivity is only enforced at the degrees of 
freedom, with no guarantee of what happens between them. Indeed, for both the initial condition and the approximations obtained at time $t=0.25$ by the post-processing method (Figure~\ref{fig:procedure_heat_post_lag}) 
and the monolithic method (Figure~\ref{fig:procedure_heat_mono_lag}) exhibit violations of the lower bounds between the spatial degrees of freedom. Notice, however, the value at the 
degrees of freedom remain positive as expected.

\begin{figure}[ht]
  \vspace{1cm}
  \centering
  \begin{subfigure}[t]{0.3\textwidth}
    \centering
  \tikz[remember picture, baseline]{\node(1L){\includegraphics[width=0.9\textwidth]{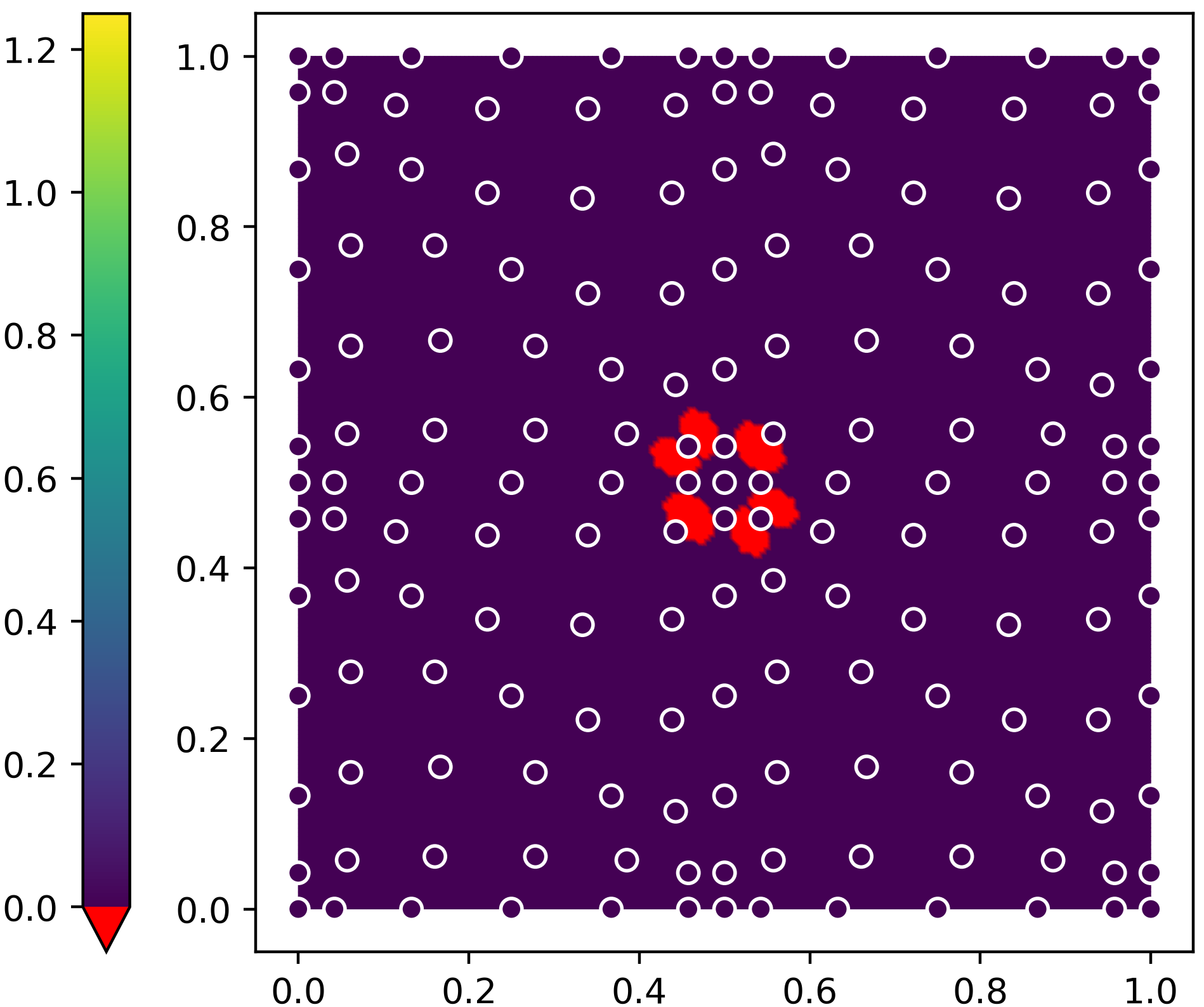}}}
  \subcaption{$t = 0.0$}
  \end{subfigure}\hspace{0.033\textwidth}%
  \begin{subfigure}[t]{0.3\textwidth}
    \centering
  \tikz[remember picture, baseline]{\node(1M){\includegraphics[width=0.8\textwidth]{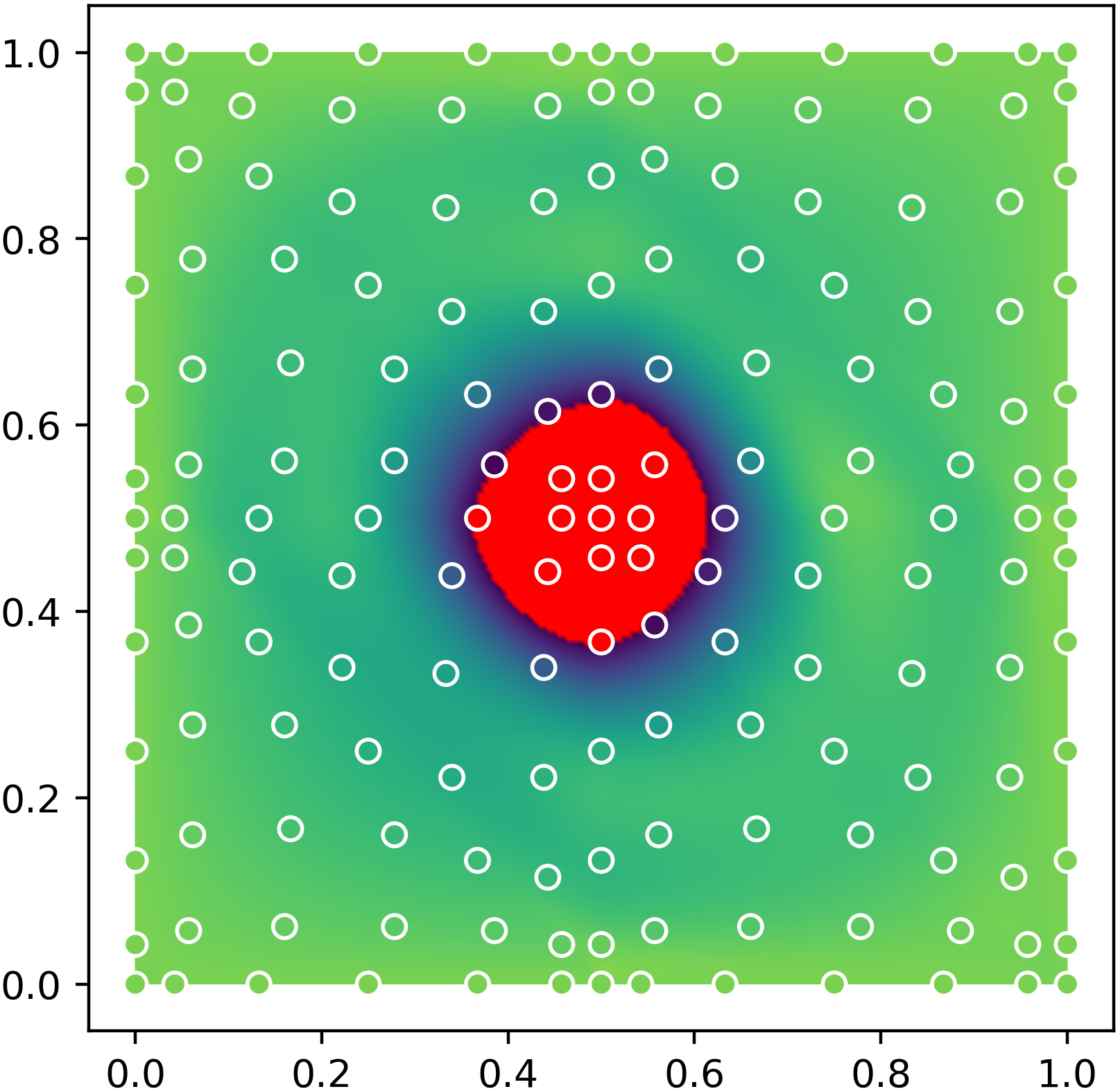}}}
  \subcaption{$t = 0.25$, pre-projection}
  \end{subfigure}\hspace{0.033\textwidth}%
  \begin{subfigure}[t]{0.3\textwidth}
    \centering
  \tikz[remember picture, baseline]{\node(1R){\includegraphics[width=0.8\textwidth]{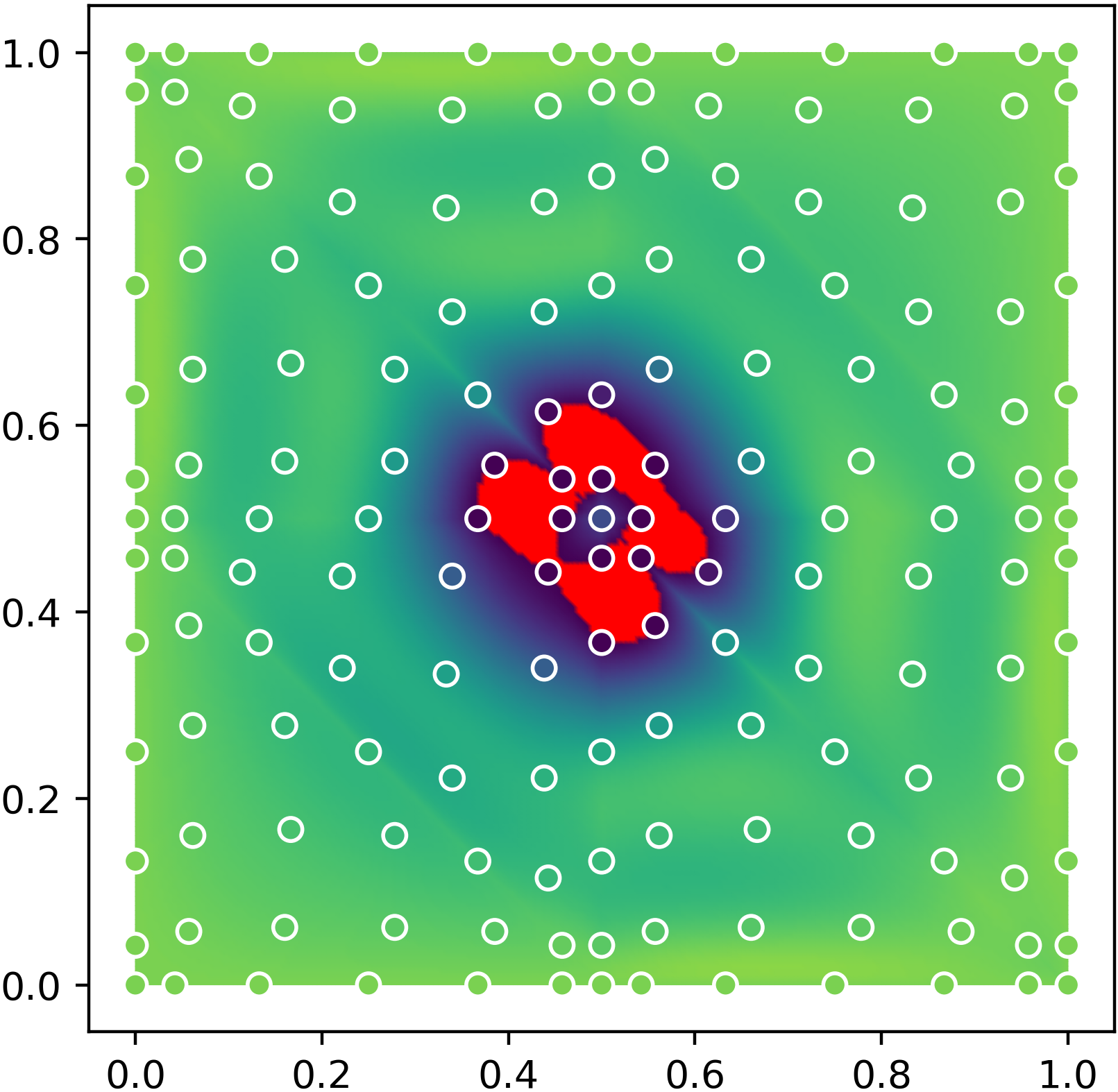}}}
  \subcaption{$t = 0.25$, projected}
  \end{subfigure}

  \caption{Snapshots of the post-processing procedure for integrating the heat equation using the $\mathcal{L}_6$-$\textrm{RIIA}(1)$-$\textrm{P}$ method.}
  \label{fig:procedure_heat_post_lag}
  \tikz[overlay,remember picture]{
      \draw[-latex,very thick] (1L.north) .. controls +(up:1cm) and +(up:1cm) .. (1M.north)
          node[midway, above, align=center]{$\mathcal{L}_6$-$\textrm{RIIA}(1)$};
      \draw[-latex,very thick] (1M.north) .. controls +(up:1cm) and +(up:1cm) .. (1R.north)
          node[midway, above, align=center]{$\mathcal{L}_6$-$\textrm{P}_{L^2}$};
  }
\end{figure}
\begin{figure}[ht]
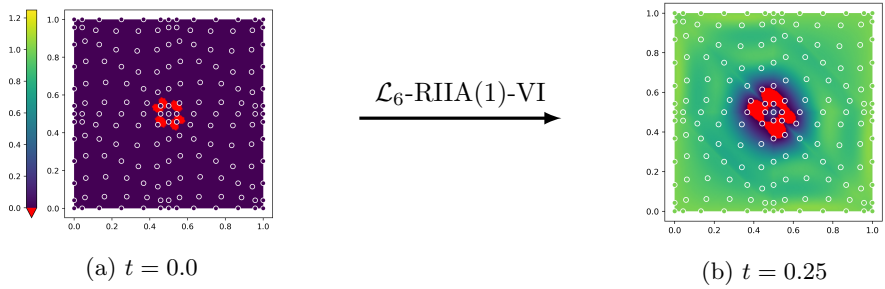

  \centering
  \begin{subfigure}[t]{0.3\textwidth}
    \centering
  \tikz[remember picture, baseline]{\node(2L){\includegraphics[width=0.9\textwidth]{L6RadauIIA1L2projTrue0.0.png}}}
  \subcaption{$t = 0.0$}
  \end{subfigure}\hspace{0.033\textwidth}%
  \hspace{0.3\textwidth}%
  \begin{subfigure}[t]{0.3\textwidth}
    \centering
  \tikz[remember picture, baseline]{\node(2R){\includegraphics[width=0.8\textwidth]{L6RadauIIA1L2projTrue_postproj_0.25.png}}}
  \subcaption{$t = 0.25$}
  \end{subfigure}

  \caption{Snapshots of the monolithic procedure for the integrating the heat equation using the $\mathcal{L}_6$-$\textrm{RIIA}(1)$-$\textrm{VI}$ method.}
  \label{fig:procedure_heat_mono_lag}
  \tikz[overlay,remember picture]{
        \draw[-latex,very thick, shorten >=1cm, shorten <=1cm] (2L.east) -- (2R.west) 
          node[midway, above, align=center]{$\mathcal{L}_6$-$\textrm{RIIA}(1)$-$\textrm{VI}$};
  }

\end{figure}

\subsubsection{Convergence Rates}

Let $\Omega = [0, 1]\times [0, 1]$. We choose $f$ and $g$ in \eqref{eq:heat_general} such that the exact solution is given by
\begin{equation}\label{eq:heat_conv_exact}
u = e^{-t} \cos(2\pi x)^2\sin(2\pi y)^2.
\end{equation}

We discretize the system with Bernstein and Lagrange finite elements of varying degree, and integrate the system to time $T = 1.0$ with our various schemes.
The mesh consists of $2\times N^2$ uniform triangular elements, and we take the time step $k = 1 / N$.
We perform the same computation with varying values of $N$ to examine the convergence behavior. The feasible sets are taken as in the previous section, with the indicated spatial basis and polynomial degree.
Results for the projective RadauIIA family are shown in Figure~\ref{fig:heat_RIIA}.
The monolithic modifications for the backward Euler method and implicit midpoint method are shown in Figure~\ref{fig:heat_monoVI}. 
The projective and monolithic modifications of the BDF family are shown in Figure~\ref{fig:heat_BDF}. We performed the same convergence tests for the projective Gauss-Legendre method 
as well and observed optimal order convergence in all cases. For the $3$-stage RadauIIA method, the absolute solver tolerance is dropped to $10^{-12}$ on the finest mesh.

Since the feasible set for the Bernstein basis contains strictly fewer functions than that for Lagrange, we expect it to have a worse best approximation but hope for the same order of accuracy. 
We do not report on tests with linear finite elements because the constrained Lagrange and Bernstein bases coincide exactly in this case. To test first and second order integration schemes, 
we instead use quadratic finite elements. For the monolithic Runge-Kutta formulations, $\mathrm{RIIA}(1)$-$\mathrm{VI}$ and $\mathrm{GL}(1)$-$\mathrm{VI}$, in the regime tested, we observe faster convergence than expected, with the error decaying at or near the optimal rate of the spatial discretization.
We expect this is because the total error is dominated by spatial rather than temporal discretization and that the rate will decay to first or second order under additional refinement.

We start the bounds-preserving $\textrm{BDF}(s)$ methods with $2^s$ steps ($s$ representing the order of the BDF method) of size $k/2^s$ of the monolithically bounds-constrained backward Euler method.
We again see optimal order convergence for all of the methods tested.

\pgfplotstableread[col sep=comma]{heat_conv_easy.csv}\loadedtable
\pgfplotstableread[col sep=comma]{heat_conv_tight.csv}\tightconvtable

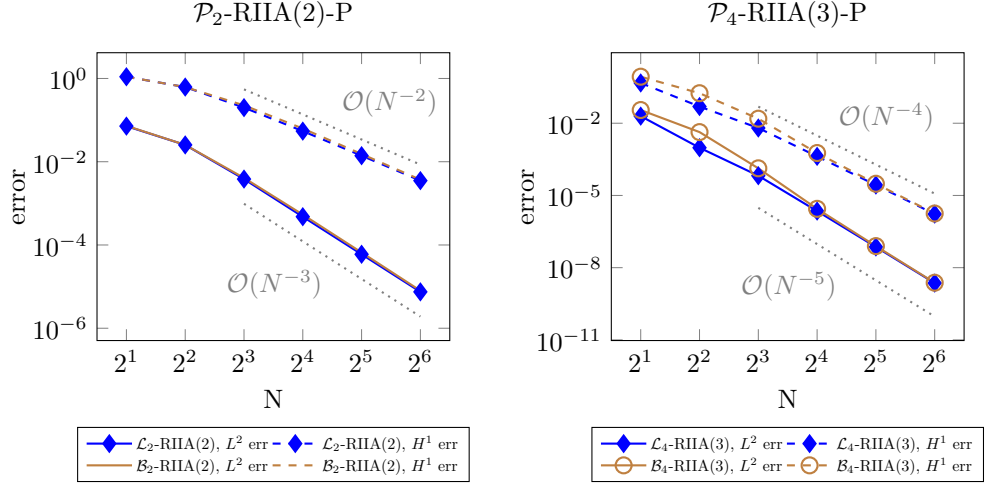
\begin{figure}[ht]
  \centering
  \begin{subfigure}[t]{0.48\textwidth}
  \centering
  \begin{tikzpicture}[baseline={(0,0)}]
  \begin{axis}[width=\textwidth, legend style={at={(0.5,-0.3)},anchor=north,nodes={scale=0.6, transform shape}, legend columns=2}, 
        title={$\mathcal{P}_2$-$\mathrm{RIIA}(2)$-$\mathrm{P}$},
    xmode=log, 
    ymode=log,
    ylabel near ticks,
    ylabel shift = {-4pt},
    xtick={2, 4,8,16,32, 64}, 
    xticklabels={$2^1$, $2^2$, $2^3$, $2^4$, $2^5$, $2^6$},
    xlabel=N,
        ylabel={error}
        ]

      \addplot[blue,mark=diamond*, mark size=3pt,  thick] table[x=Nspat, y=Lagrange2RadauIIA2L2projTrueL2err, col sep=comma]\loadedtable;
      \addlegendentry[thick] {$\mathcal{L}_2$-$\mathrm{RIIA}(2)$, $L^2$ err}

      \addplot[blue,mark=diamond*, mark size=3pt,  thick, dashed, mark options={solid}] table[x=Nspat, y=Lagrange2RadauIIA2L2projTrueH1err, col sep=comma] \loadedtable;
      \addlegendentry {$\mathcal{L}_2$-$\mathrm{RIIA}(2)$, $H^1$ err}
      

      \addplot[brown,mark=0, mark size=3pt,  thick] table[x=Nspat, y=Bernstein2RadauIIA2L2projTrueL2err, col sep=comma] \loadedtable;
      \addlegendentry[thick] {$\mathcal{B}_2$-$\mathrm{RIIA}(2)$, $L^2$ err}
      
      \addplot[brown,mark=0, mark size=3pt,  thick, dashed, mark options={solid}] table[x=Nspat, y=Bernstein2RadauIIA2L2projTrueH1err, col sep=comma] \loadedtable;
      \addlegendentry{$\mathcal{B}_2$-$\mathrm{RIIA}(2)$, $H^1$ err}


      \addplot[gray, dotted, thick, domain=8:64] {35/x^2} node[pos=0.5, anchor=south west] {$\mathcal{O}(N^{-2})$};
      \addplot[gray, dotted, thick, domain=8:64] {0.5/x^3} node[pos=0.5, anchor=north east] {$\mathcal{O}(N^{-3})$};

  \end{axis}

  \end{tikzpicture}
  \end{subfigure}\hspace{0.04\textwidth}%
  \begin{subfigure}[t]{0.48\textwidth}
  \centering
  \begin{tikzpicture}[baseline={(0,0)}]
  \begin{axis}[width=\textwidth, legend style={at={(0.5,-0.3)},anchor=north,nodes={scale=0.6, transform shape}, legend columns=2},
          title={$\mathcal{P}_4$-$\mathrm{RIIA}(3)$-$\mathrm{P}$},
    xmode=log, 
    ymode=log,
    ylabel near ticks,
    ylabel shift = {-6pt},
    xtick={2, 4,8,16,32, 64, 128}, 
    xticklabels={$2^1$, $2^2$, $2^3$, $2^4$, $2^5$, $2^6$, $2^7$},
    xlabel=N,
        ylabel={error}]

      \addplot[blue,mark=diamond*, mark size=3pt,  thick] table[x=Nspat, y=Lagrange4RadauIIA3L2projTrueL2err, col sep=comma] \tightconvtable;
      \addlegendentry[thick] {$\mathcal{L}_4$-$\mathrm{RIIA}(3)$, $L^2$ err}
      
      \addplot[blue,mark=diamond*, mark size=3pt,  thick, dashed, mark options={solid}] table[x=Nspat, y=Lagrange4RadauIIA3L2projTrueH1err, col sep=comma] \tightconvtable;
      \addlegendentry[thick] {$\mathcal{L}_4$-$\mathrm{RIIA}(3)$, $H^1$ err}
      

      \addplot[brown,mark=o, mark size=3pt,  thick] table[x=Nspat, y=Bernstein4RadauIIA3L2projTrueL2err, col sep=comma] \tightconvtable;
      \addlegendentry[thick] {$\mathcal{B}_4$-$\mathrm{RIIA}(3)$, $L^2$ err}
          
      \addplot[brown,mark=o, mark size=3pt,  thick, dashed, mark options={solid}] table[x=Nspat, y=Bernstein4RadauIIA3L2projTrueH1err, col sep=comma] \tightconvtable;
      \addlegendentry[thick] {$\mathcal{B}_4$-$\mathrm{RIIA}(3)$, $H^1$ err}
            

      \addplot[gray, dotted, thick, domain=8:64] {200/x^4} node[pos=0.4, anchor=south west] {$\mathcal{O}(N^{-4})$};
      \addplot[gray, dotted, thick, domain=8:64] {0.1/x^5} node[pos=0.5, anchor=north east] {$\mathcal{O}(N^{-5})$};

  \end{axis}
  \end{tikzpicture}
  \end{subfigure}
\caption{Convergence plots for projective methods. $L^2$ (solid) and $H^1$ (dotted) errors in the approximation of $u(1.0)$ using a uniform $N\times N$ mesh and $k = 1/N$.} 
\label{fig:heat_RIIA}
\end{figure}


\begin{figure}[ht]
  \centering
  \begin{subfigure}[t]{0.48\textwidth}
    \centering
  \begin{tikzpicture}
  \begin{axis}[width=\textwidth, legend style={at={(0.5,-0.3)},anchor=north,nodes={scale=0.6, transform shape}, legend columns=2},
          title={$\mathcal{P}_2$-$\mathrm{RIIA}(1)$-$\mathrm{VI}$},
    xmode=log, 
    ymode=log,
    ylabel near ticks,
    ylabel shift = {-4pt},
    xtick={2, 4,8,16,32, 64, 128}, 
    xticklabels={$2^1$, $2^2$, $2^3$, $2^4$, $2^5$, $2^6$, $2^7$},
    xlabel=N,
        ylabel={error}
        ]
      \addplot[blue,mark=diamond*, mark size=3pt,  thick] table[x=Nspat, y=Lagrange2RadauIIA1monoTrueL2err, col sep=comma] \loadedtable;
      \addlegendentry[thick] {$\mathcal{L}_2$-$\mathrm{RIIA}(1)$, $L^2$ err}

      \addplot[blue,mark=diamond*, mark size=3pt,  thick, dashed, mark options={solid}] table[x=Nspat, y=Lagrange2RadauIIA1monoTrueH1err, col sep=comma] \loadedtable;
      \addlegendentry[thick] {$\mathcal{L}_2$-$\mathrm{RIIA}(1)$, $H^1$ err}
      

      \addplot[brown,mark=o, mark size=3pt, thick] table[x=Nspat, y=Bernstein2RadauIIA1monoTrueL2err, col sep=comma] \loadedtable;
      \addlegendentry[thick] {$\mathcal{B}_2$-$\mathrm{RIIA}(1)$, $L^2$ err}

      \addplot[brown,mark=o, mark size=3pt,  thick, dashed,mark options={solid}] table[x=Nspat, y=Bernstein2RadauIIA1monoTrueH1err, col sep=comma]\loadedtable;
      \addlegendentry[thick] {$\mathcal{B}_2$-$\mathrm{RIIA}(1)$, $H^1$ err}

      \addplot[gray, dotted, thick, domain=8:64] {50/x^2} node[pos=0.5, anchor=south west] {$\mathcal{O}(N^{-2})$};
      \addplot[gray, dotted, thick, domain=8:64] {0.6/x^3} node[pos=0.5, anchor=north east] {$\mathcal{O}(N^{-3})$};

  \end{axis}

  \end{tikzpicture}
  \end{subfigure}\hspace{0.04\textwidth}%
  \begin{subfigure}[t]{0.48\textwidth}
    \centering
  \begin{tikzpicture}
  \centering
  \begin{axis}[width=\textwidth, legend style={at={(0.5,-0.3)},anchor=north,nodes={scale=0.6, transform shape}, legend columns=2},
          title={$\mathcal{P}_2$-$\mathrm{GL}(1)$-$\mathrm{VI}$},
    xmode=log, 
    ymode=log,
    ylabel near ticks,
    ylabel shift = {-4pt},
    xtick={2, 4,8,16,32, 64, 128}, 
    xticklabels={$2^1$, $2^2$, $2^3$, $2^4$, $2^5$, $2^6$, $2^7$},
    xlabel=N,
        ylabel={error}]
      \addplot[blue,mark=diamond*, mark size=3pt,  thick] table[x=Nspat, y=Lagrange2GaussLegendre1monoTrueL2err, col sep=comma] \loadedtable;
      \addlegendentry[thick] {$\mathcal{L}_2$-$\mathrm{GL}(1)$, $L^2$ err}
      
      \addplot[blue,mark=diamond*, mark size=3pt,  thick, dashed, mark options={solid}] table[x=Nspat, y=Lagrange2GaussLegendre1monoTrueH1err, col sep=comma] \loadedtable;
      \addlegendentry[thick] {$\mathcal{L}_2$-$\mathrm{GL}(1)$, $H^1$ err}

      \addplot[brown,mark=o, mark size=3pt, thick] table[x=Nspat, y=Bernstein2GaussLegendre1monoTrueL2err, col sep=comma] \loadedtable;
      \addlegendentry[thick] {$\mathcal{B}_2$-$\mathrm{GL}(1)$, $L^2$ err}

      \addplot[brown,mark=o, mark size=3pt,  thick, dashed,mark options={solid}] table[x=Nspat, y=Bernstein2GaussLegendre1monoTrueH1err, col sep=comma]\loadedtable;
      \addlegendentry[thick] {$\mathcal{B}_2$-$\mathrm{GL}(1)$, $H^1$ err}

      \addplot[gray, dotted, thick, domain=8:64] {50/x^2} node[pos=0.5, anchor=south west] {$\mathcal{O}(N^{-2})$};
      \addplot[gray, dotted, thick, domain=8:64] {0.6/x^3} node[pos=0.5, anchor=north east] {$\mathcal{O}(N^{-3})$};

  \end{axis}
  \end{tikzpicture}
  \end{subfigure}
\caption{Convergence plots for monolithic methods. $L^2$ (solid) and $H^1$ (dotted) error in the approximation of $u(1.0)$ using a uniform $N\times N$ mesh and $k = 1/N$.} 
\label{fig:heat_monoVI}
\end{figure}


\pgfplotstableread[col sep=comma]{heat_conv_BDF_easy.csv}\bdftable

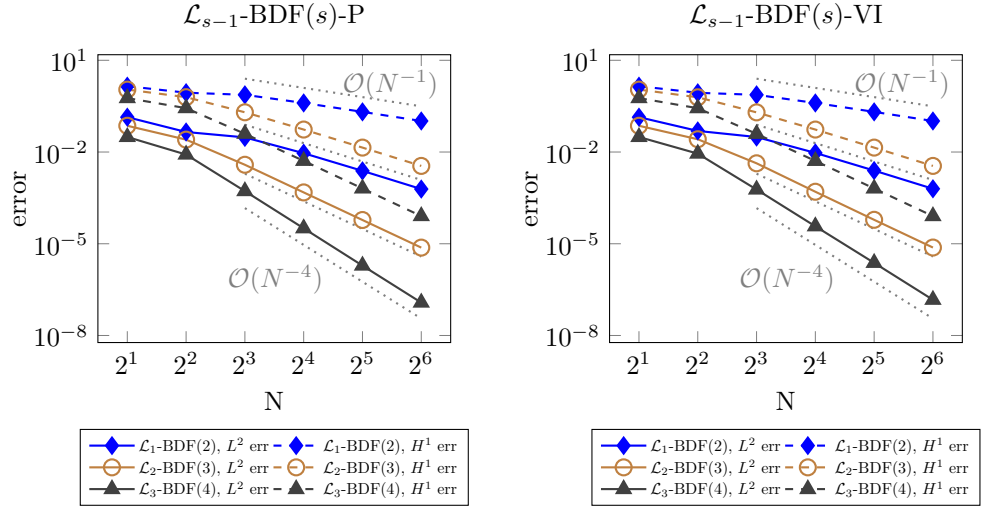
\begin{figure}[ht]
  \centering
  \begin{subfigure}[t]{0.48\textwidth}
    \centering
  \begin{tikzpicture}
  \begin{axis}[width=\textwidth, legend style={at={(0.5,-0.3)},anchor=north,nodes={scale=0.6, transform shape}, legend columns=2},
      title={$\mathcal{L}_{s-1}$-$\mathrm{BDF}(s)$-$\mathrm{P}$},
    xmode=log, 
    ymode=log,
    ylabel near ticks,
    ylabel shift = {-4pt},
    xtick={2, 4,8,16,32, 64, 128}, 
    xticklabels={$2^1$, $2^2$, $2^3$, $2^4$, $2^5$, $2^6$, $2^7$},
    xlabel=N,
        ylabel={error}]
    \addplot[blue,mark=diamond*, mark size=3pt, thick] table[x=Nspat, y=Lagrange1BDF2L2projTrueL2err, col sep=comma] \bdftable;
    \addlegendentry[thick] {$\mathcal{L}_1$-$\mathrm{BDF}(2)$, $L^2$ err}
    \addplot[blue,mark=diamond*, mark size=3pt,  thick, dashed,mark options={solid}] table[x=Nspat, y=Lagrange1BDF2L2projTrueH1err, col sep=comma]\bdftable;
    \addlegendentry[thick] {$\mathcal{L}_1$-$\mathrm{BDF}(2)$, $H^1$ err}

    \addplot[brown,mark=o, mark size=3pt, thick] table[x=Nspat, y=Lagrange2BDF3L2projTrueL2err, col sep=comma] \bdftable;
    \addlegendentry[thick] {$\mathcal{L}_2$-$\mathrm{BDF}(3)$, $L^2$ err}
    \addplot[brown,mark=o, mark size=3pt,  thick, dashed,mark options={solid}] table[x=Nspat, y=Lagrange2BDF3L2projTrueH1err, col sep=comma]\bdftable;
    \addlegendentry[thick] {$\mathcal{L}_2$-$\mathrm{BDF}(3)$, $H^1$ err}
    \addplot[darkgray,mark=triangle*, mark size=3pt,  thick] table[x=Nspat, y=Lagrange3BDF4L2projTrueL2err, col sep=comma] \bdftable;
    \addlegendentry[thick] {$\mathcal{L}_3$-$\mathrm{BDF}(4)$, $L^2$ err}
     \addplot[darkgray,mark=triangle*, mark size=3pt,  thick, dashed, mark options={solid}] table[x=Nspat, y=Lagrange3BDF4L2projTrueH1err, col sep=comma] \bdftable;
    \addlegendentry[thick] {$\mathcal{L}_3$-$\mathrm{BDF}(4)$, $H^1$ err}



    \addplot[gray, dotted, thick, domain=8:64] {20/x^1} node[pos=0.5, anchor=south west, yshift={-6pt}] {$\mathcal{O}(N^{-1})$};
    \addplot[gray, dotted, thick, domain=8:64] {5/x^2};
    \addplot[gray, dotted, thick, domain=8:64] {1/x^3};
    \addplot[gray, dotted, thick, domain=8:64] {0.6/x^4} node[pos=0.5, anchor=north east, yshift={3pt}] {$\mathcal{O}(N^{-4})$};

  \end{axis}
  \end{tikzpicture}
  \end{subfigure}\hspace{0.04\textwidth}%
  \begin{subfigure}[t]{0.48\textwidth}
    \centering
  \begin{tikzpicture}
  \centering
  \begin{axis}[width=\textwidth, legend style={at={(0.5,-0.3)},anchor=north,nodes={scale=0.6, transform shape}, legend columns=2},
            title={$\mathcal{L}_{s - 1}$-$\mathrm{BDF}(s)$-$\mathrm{VI}$},
    xmode=log, 
    ymode=log,
    ylabel near ticks,
    ylabel shift = {-4pt},
    xtick={2, 4,8,16,32, 64, 128}, 
    xticklabels={$2^1$, $2^2$, $2^3$, $2^4$, $2^5$, $2^6$, $2^7$},
    xlabel=N,
        ylabel={error}]

      \addplot[blue,mark=diamond*, mark size=3pt, thick] table[x=Nspat, y=Lagrange1BDF2monoTrueL2err, col sep=comma] \bdftable;
      \addlegendentry[thick] {$\mathcal{L}_1$-$\mathrm{BDF}(2)$, $L^2$ err}
      \addplot[blue,mark=diamond*, mark size=3pt,  thick, dashed,mark options={solid}] table[x=Nspat, y=Lagrange1BDF2monoTrueH1err, col sep=comma]\bdftable;
      \addlegendentry[thick] {$\mathcal{L}_1$-$\mathrm{BDF}(2)$, $H^1$ err}

      \addplot[brown,mark=o, mark size=3pt, thick] table[x=Nspat, y=Lagrange2BDF3monoTrueL2err, col sep=comma] \bdftable;
      \addlegendentry[thick] {$\mathcal{L}_2$-$\mathrm{BDF}(3)$, $L^2$ err}
      \addplot[brown,mark=o, mark size=3pt,  thick, dashed,mark options={solid}] table[x=Nspat, y=Lagrange2BDF3monoTrueH1err, col sep=comma]\bdftable;
      \addlegendentry[thick] {$\mathcal{L}_2$-$\mathrm{BDF}(3)$, $H^1$ err}

      \addplot[darkgray,mark=triangle*, mark size=3pt,  thick] table[x=Nspat, y=Lagrange3BDF4monoTrueL2err, col sep=comma] \bdftable;
      \addlegendentry[thick] {$\mathcal{L}_3$-$\mathrm{BDF}(4)$, $L^2$ err}
      \addplot[darkgray,mark=triangle*, mark size=3pt,  thick, dashed, mark options={solid}] table[x=Nspat, y=Lagrange3BDF4monoTrueH1err, col sep=comma] \bdftable;
      \addlegendentry[thick] {$\mathcal{L}_3$-$\mathrm{BDF}(4)$, $H^1$ err}


      \addplot[gray, dotted, thick, domain=8:64] {20/x^1} node[pos=0.5, anchor=south west, yshift={-6pt}] {$\mathcal{O}(N^{-1})$};
      \addplot[gray, dotted, thick, domain=8:64] {5/x^2};
      \addplot[gray, dotted, thick, domain=8:64] {1/x^3};
      \addplot[gray, dotted, thick, domain=8:64] {0.6/x^4} node[pos=0.5, anchor=north east, yshift={3pt}] {$\mathcal{O}(N^{-4})$};

  \end{axis}
  \end{tikzpicture}
  \end{subfigure}%

\caption{Convergence plots for projective methods (left) and monolithic methods (right). $L^2$ (solid) and $H^1$ (dotted) error in the approximation of $u(1.0)$ using a uniform $N\times N$ mesh and $k = 1/N$.} 
\label{fig:heat_BDF}
\end{figure}

\subsubsection{Timing Results}\pgfplotstableread[col sep=comma]{heat_timing_rana_new.csv}\heattiming

To investigate how the time required for the post-processing procedure scales with the polynomial degree and number of internal stages, we fix a triangular mesh consisting of $32,768$ uniform elements parameterized using the Bernstein basis. 
We record the time it takes to take one step of size $1 / 128$ with constrained $L^2$ post-processing. The results are shown in Figure~\ref{fig:heat_timing_Bernstein}.

We solve the linear system arising in RadauIIA discretization with
GMRES~\cite{saad1986gmres} and a relative tolerance of $10^{-8}$.
We precondition this system using the stage-segregated preconditioner of~\cite{masud2021new}, approximating the inverses of the diagonal blocks with one sweep of PETSc's geometric algebraic multigrid.
Then, we perform the constrained $L^2$ projection with the reduced space Newton scheme as described above.
As expected, we observe that the time required for the projection is
independent of the order of the underlying time-stepping scheme, depending primarily on the 
polynomial degree.

As the underlying problem becomes more challenging, 
the solution to the stage-coupled system will require additional time, while the constrained post-processing is independent of the underlying problem.
We repeated these tests with the polynomials parameterized using the Lagrange basis and observed no meaningful differences from the Bernstein case.

\begin{figure}[ht]
  \centering
  \begin{tikzpicture}
  \begin{axis}[width=0.6\textwidth, height=0.4\textwidth, legend style={at={(1.05, 0.5)}, anchor=west,nodes={scale=0.8, transform shape}, legend columns=1},
          title={$\mathcal{B}_p$-$\mathrm{RIIA}(s)$-$\mathrm{P}$},
    ymode=log,
    ylabel near ticks,
    xlabel near ticks,
    xtick={1, 2, 3, 4}, 
    xlabel=$s$,
    ymin=0.025,
    ymax=150,
    ylabel={time (s)}
        ]

      \addplot[blue,mark=diamond*, mark size=3pt,  thick] table[x=s, y=B1RadauIIAL2projTrueSolverTime, col sep=comma] \heattiming;
      \addlegendentry[thick] {$p = 1$}
      
      \addplot[brown,mark=*, mark size=3pt, thick] table[x=s, y=B2RadauIIAL2projTrueSolverTime, col sep=comma] \heattiming;
      \addlegendentry[thick] {$p = 2$}
      
      \addplot[darkgray,mark=triangle*, mark size=3pt, thick] table[x=s, y=B3RadauIIAL2projTrueSolverTime, col sep=comma] \heattiming;
      \addlegendentry[thick] {$p = 3$}

      \addplot[red,mark=square*, mark size=3pt,  thick] table[x=s, y=B4RadauIIAL2projTrueSolverTime, col sep=comma] \heattiming;
      \addlegendentry[thick] {$p = 4$}

      
      \addplot[blue,mark=diamond, mark size=3pt,  thick, dashed, mark options={solid}, forget plot] table[x=s, y=B1RadauIIAL2projTrueProjectorTime, col sep=comma] \heattiming;
      \addplot[brown,mark=o, mark size=3pt,  thick, dashed,mark options={solid}, forget plot] table[x=s, y=B2RadauIIAL2projTrueProjectorTime, col sep=comma]\heattiming;
      \addplot[darkgray,mark=triangle, mark size=3pt,  thick, dashed,mark options={solid}, forget plot] table[x=s, y=B3RadauIIAL2projTrueProjectorTime, col sep=comma]\heattiming;
      \addplot[red,mark=square, mark size=3pt,  thick, dashed,mark options={solid}, forget plot] table[x=s, y=B4RadauIIAL2projTrueProjectorTime, col sep=comma] \heattiming;

      

  \end{axis}
  \end{tikzpicture}

\caption{Time required to solve the stage-coupled system (solid) and perform an $L^2$ post-processing (dashed) for taking one step of size $k = 1 / 128$ for the heat equation on a $32,768$ element triangular mesh with the elements parameterized in the Bernstein basis.  We use degree $p$ spatial finite elements, and the $s$-stage RadauIIA Runge-Kutta method.} 
\label{fig:heat_timing_Bernstein}
\end{figure}
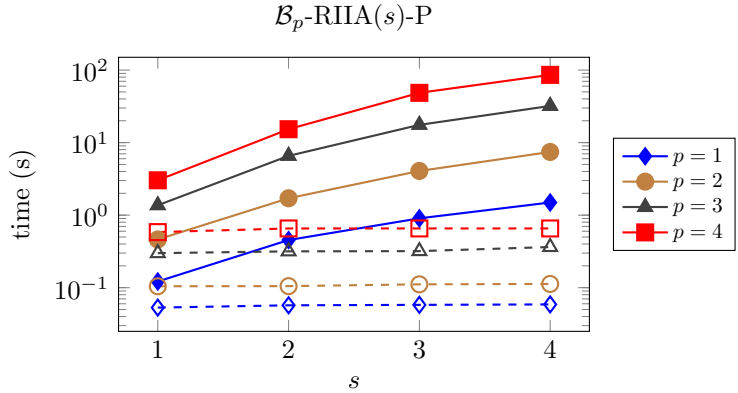

\subsection{The Advection equation}
\label{sec:advec}
We now turn our attention to the advection problem~\eqref{eq:advec_eq}-\eqref{eq:advec_ic}. 
We take $\Omega = [0, 1]\times [0, 1]$.  For the flow field, we use the swirling deformation flow test case of~\cite{leveque1996high},
in which
\begin{equation}\label{eq:advec_weak}
        \mathbf{v} = g(t)\begin{bmatrix} \sin^2(\pi x)\sin(2\pi y)\\ -\sin^2(\pi y)\sin(2\pi x)\end{bmatrix},
\end{equation}
with $g(t) = \cos(\pi t / T)$. The flow reverses itself 
at time $t = T/2$, so the initial condition is recovered exactly at time $T$, providing an easy way to check the quality of a numerical solution. We take the final time to be $T = 1.5$.

We test our methods with the cosine--cone--slotted-cylinder initial condition~\cite{leveque1996high}, with a vertical shift of one (yielding a solution satisfying $1.0\leq u\leq 2.0$).
The discrete initial condition is obtained by the $L^2$ bounds-constrained projection into the Bernstein basis, yielding a uniformly bounds-constrained finite element approximation.
We then interpolate this function into the Lagrange basis for use as the initial condition.

Many techniques have been developed to accurately simulate advective flows. Slope limiters, flux-corrected transport, and 
other specially designed spatio-temporal discretizations have found great success. 
Here, we use a discontinuous Galerkin method with an upwind flux, but do not implement limiters or other stabilization techniques. 
We enforce that the range of the approximation is contained in $[1, 2]$ for every time $t^n$ using a post-processing or monolithic technique.

For any function $f$ defined on the finite element mesh, $\mathcal{T}_h$, let $f_+$ and $f_-$ denote the value of $f$ on the upwind and downwind side of a cell facet, respectively, relative to the flow field $\mathbf{v}$. Let $\mathbf{n}$ denote the unit outward normal vector, and $\Gamma_\text{int}$ the set of interior facets of the mesh.
After integration by parts, the resulting weak form is to seek $u\in\mathcal{V}_h = \left\{u\in L^2(\Omega): u|_{T} \in \mathcal{P}_k(T)\quad \forall T\in \mathcal{T}_h\right\}$ such that 
\begin{multline}
(u_t, w)_\Omega = (u, \textrm{div} (w\mathbf{v}))_\Omega - (u_+, w_+ \mathbf{v}\cdot \mathbf{n}_+ + w_- \mathbf{v}\cdot \mathbf{n}_-)_{\Gamma_\text{int}}\\ - ((\mathbf{v}\cdot \mathbf{n}) u_\text{in}, w)_{\Gamma_\text{inflow}} - ((\mathbf{v}\cdot \mathbf{n}) u, w)_{\Gamma\backslash\Gamma_\text{inflow}}
\end{multline}
for all test functions $w\in \mathcal{V}_h$. Here, the subscripts $\Omega$, $\Gamma_\text{int}$ and $\Gamma\backslash \Gamma_\text{int}$ denote the domain of integration.

Let $\mathcal{T}_h$ denote the triangulation of $\Omega$ obtained by taking a $100\times 100$ square mesh, with each cell subdivided into two triangles. We use spatial elements of degree $2$. We approximate the flow using the Lagrange basis in space, and a variety of temporal integration schemes.
We compare our several bounds-constrained time steppers methods to a very common baseline -- the $\textrm{SSPRK}(3, 3)$ method of~\cite{shu1988efficient}.
This explicit 3-stage method is widely used for hyperbolic problems and, provided the CFL condition is respected, gives accurate solutions.
It is conditionally stable (but not bounds-preserving) for this problem without limiters or other postprocessing techniques.
We compare it to postprocessing and monolithic bounds-preserving schemes, as described above.
In our first set of experiments, we take the feasible set as $\mathcal{K}_h = \mathcal{J}_h^{\mathcal{L}_2, [1, 2]}$ with no interelement continuity enforced. 
Our numerical experiment is straightforward -- for each of several time-stepping schemes, we select the largest time step 
so that $\frac{||u_h(T) - u(T)||_{L^2}}{||u(T)||_{L^2}} \leq 0.05$ -- no more than 5\% relative error -- and report the total time taken to integrate from $0$ to $T$.

In all of these experiments, we solve our linear systems with GMRES using a Euclidean tolerance of $10^{-8}$.
For single-stage linear systems, we precondition this system with $\ell$-AIR~\cite{manteuffel2018nonsymmetric} implemented using {\ttfamily boomeramg} from the Hypre package~\cite{hypre,yang2002boomeramg}.
For the multi-stage linear systems arising in GL(s) and RadauIIA(s) for $s > 1$, we use the stage-segregating Rana preconditioner with the same $\ell$-AIR configuration used for single-stage systems on each diagonal block.
The postprocessing  is performed by solving the constrained $L^2$ projection with \texttt{vinewtonrsls} using Jacobians inverted by SOR-preconditioned conjugate gradients.

We give results for the SSPRK(3,3) and projective schemes based on SSPRK(3, 3), Gauss-Legendre, RadauIIA, and BDF schemes in Table~\ref{tab:time_advection_proj_loose}.
Here, for each method, we report the time step used, the time spent in linear solvers (Solver) and in the postprocessing projection (Update).
Perhaps surprisingly, the excellent stability and accuracy with large time steps, aided by effective preconditioning,
makes the GL(2) and GL(3) methods more efficient than SSPRK(3, 3). While taking slightly longer, the higher-order Radau methods are also quite competitive. 
We also show the the results for the monolithic formulations of GL(1) and BDF(2) in Table~\ref{tab:time_advection_mono_loose}.

For methods using $L^2$ post-processing, we see that, in all cases, the time required for the 
solution of the discrete system dominates the time required for the projection.
While there is no issue with stability, the low order and diffusive nature of backward Euler lead to  poor accuracy.
The Gauss-Legendre family, along with the higher-order RadauIIA and BDF methods, fare much better.

We also consider a variation of this experiment that preserves mass as well as bounds, taking the feasible set to be
\begin{equation*}
  \mathcal{K}_h = \mathcal{J}_h^{\mathcal{L}_2, [1, 2]}\cap \left\{u\in L^2(\Omega) : \int_\Omega u \mathrm{dx} = \int_\Omega u_0 \mathrm{dx}\right\}.
\end{equation*}
The results are shown in Table~\ref{tab:time_advection_mass_proj}. We see that while the time required for the projection increases, the largest step size yielding the desired 
error does not change drastically, and the time spent in the nonlinear projection remains substantially smaller than the time required to advance the unconstrained system. Results for the higher order Gauss-Legendre and RadauIIA methods 
are similar to those without the mass constraint, though with the increase in the time required for the more heavily constrained projection. We do not include these results here.

We see that for some of the methods chosen, the solution of the fully implicit discretized system and the constrained $L^2$ projection 
takes less time than the (unconstrained) solution with $\textrm{SSPRK}(3,3)$. The use of implicit methods allows for accuracy and stability with larger time steps, 
resulting in, in some cases, accurate results with less computational time. The monolithically constrained systems take significantly 
longer than the $\textrm{SSPRK}(3,3)$ method, but avoid computing intermediate unconstrained approximations.

For slightly larger time steps, the unconstrained $\textrm{SSPRK}(3,3)$ method becomes unstable. However, as soon as the time step is small enough that a solution is obtained, the relative error is already smaller than the given tolerance at $0.043$. 
\begin{table}[ht]
  \begin{center}
  \begin{tabular}{l l l l l l}
  Integrator & $k$ & Solver (s) & Update (s)\\
  \hline
  $\textrm{SSPRK}(3,3)$                    & $T / 883$ &  $61.16$ & -\\
  $\textrm{SSPRK}(3,3)$-$\textrm{P}$ & $T / 950$ &  $67.81$ & $47.28$\\ \hline
  $\textrm{GL}(1)$-$\textrm{P}$      & $T / 149 $ & $68.53$ & $8.84$\\
  $\textrm{GL}(2)$-$\textrm{P}$      & $T / 42$ & $33.27$ & $2.86$\\
  $\textrm{GL}(3)$-$\textrm{P}$      & $T / 23$ & $47.78$ & $1.74$\\
  $\textrm{GL}(4)$-$\textrm{P}$      & $T / 15$ & $62.58$ & $1.16$\\
  $\textrm{GL}(5)$-$\textrm{P}$      & $T / 11$ & $81.03$ & $0.96$\\
  $\textrm{RIIA}(1)$-$\textrm{P}$    & $T / 3470$ & $1129.21$ & $119.98$\\
  $\textrm{RIIA}(2)$-$\textrm{P}$    & $T / 93$ & $71.48$ & $6.02$\\
  $\textrm{RIIA}(3)$-$\textrm{P}$    & $T / 36$ & $72.17$ & $2.51$\\
  $\textrm{RIIA}(4)$-$\textrm{P}$    & $T / 21$ & $84.84$ & $1.60$\\
  $\textrm{RIIA}(5)$-$\textrm{P}$    & $T / 15$ & $108.07$ & $1.28$\\ \hline
  $\textrm{BDF}(2)$-$\textrm{P}$     & $T / 203$ & $91.44$ & $11.08$\\
  $\textrm{BDF}(3)$-$\textrm{P}$     & $T / 677$ & $244.04$ &$ 33.76$\\
  \end{tabular}
\caption{Time to solution for $L^2$ post-processing of the numerical solution to the pure advection problem, with the elements parameterized in the Lagrange basis.}
\label{tab:time_advection_proj_loose}
\end{center}
\end{table}
\begin{table}[ht]
  \begin{center}
  \begin{tabular}{l l l l l l}
  Integrator & $k$ & Solver Time (s)\\
  \hline
  $\textrm{GL}(1)$-${VI}$      & $T / 187$ & $228.36$ \\
  $\textrm{BDF}(2)$-${VI}$     & $T / 321$ & $324.91$ \\
  \end{tabular}
\caption{Time to solution for the monolithically-constrained numerical solution to the pure advection problem, with the elements parameterized in the Lagrange basis.}
\label{tab:time_advection_mono_loose}
\end{center}
\end{table}
\begin{table}[ht]
  \begin{center}
  \begin{tabular}{l l l l l l}
  Integrator & $k$ & Solver (s) & Update (s)\\
  \hline
  $\textrm{SSPRK}(3,3)$-$\textrm{P}$ & $T / 998$ & $73.90$ & $78.92$\\
  $\textrm{GL}(1)$-$\textrm{P}$      & $T / 153$ & $68.70$  &$32.78$ \\
  $\textrm{GL}(2)$-$\textrm{P}$      & $T / 42$  & $33.43$ & $16.48$\\
  $\textrm{RIIA}(2)$-$\textrm{P}$    & $T / 97 $ & $72.39$ & $23.71$\\
  $\textrm{BDF}(2)$-$\textrm{P}$     & $T / 205$ & $92.97$ & $30.64$\\
  $\textrm{BDF}(3)$-$\textrm{P}$     & $T / 678$ & $246.26$ & $134.01$\\
  \end{tabular}
\caption{Time to solution for $L^2$, mass-constrained, post-processing of the numerical solution to the pure advection problem, with the elements parameterized in the Lagrange basis.}
\label{tab:time_advection_mass_proj}
\end{center}
\end{table}

\begin{figure}
\begin{subfigure}{.33\textwidth}
\centering
\rotatebox{90}{\hspace{1em} $\mathcal{L}_2$-$\mathrm{SSPRK}(3,3)$}%
\includegraphics[trim={0cm 0cm 0cm 0cm},clip,width=0.9\textwidth] {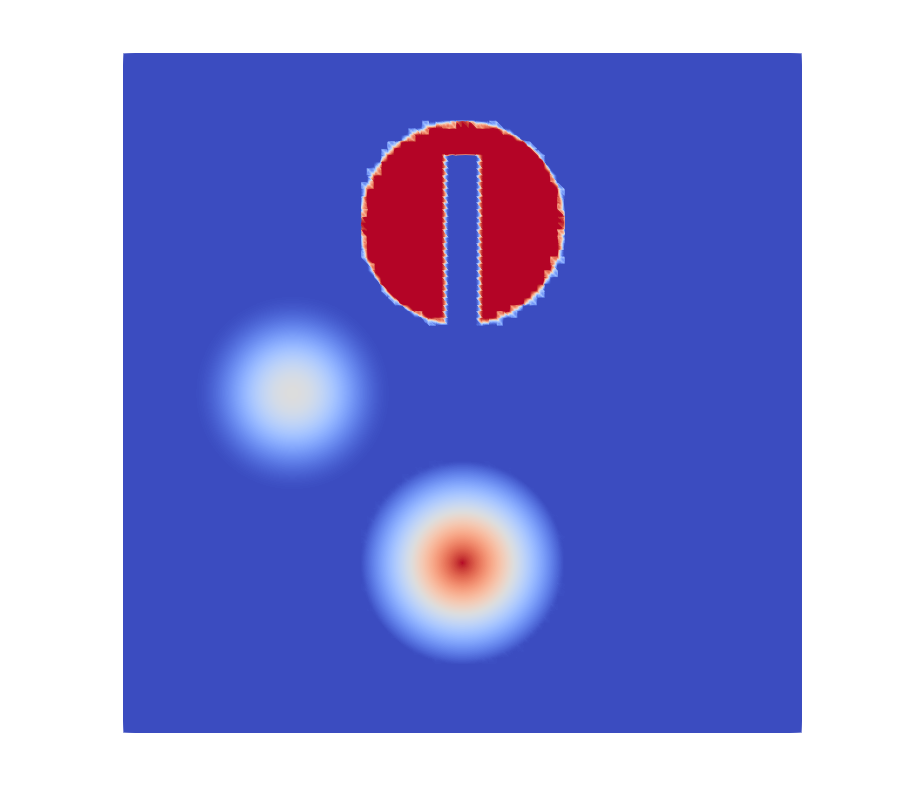}
\end{subfigure}%
\begin{subfigure}{.33\textwidth}
\centering
\includegraphics[trim={0cm 0cm 0cm 0cm},clip,width=0.9\textwidth] {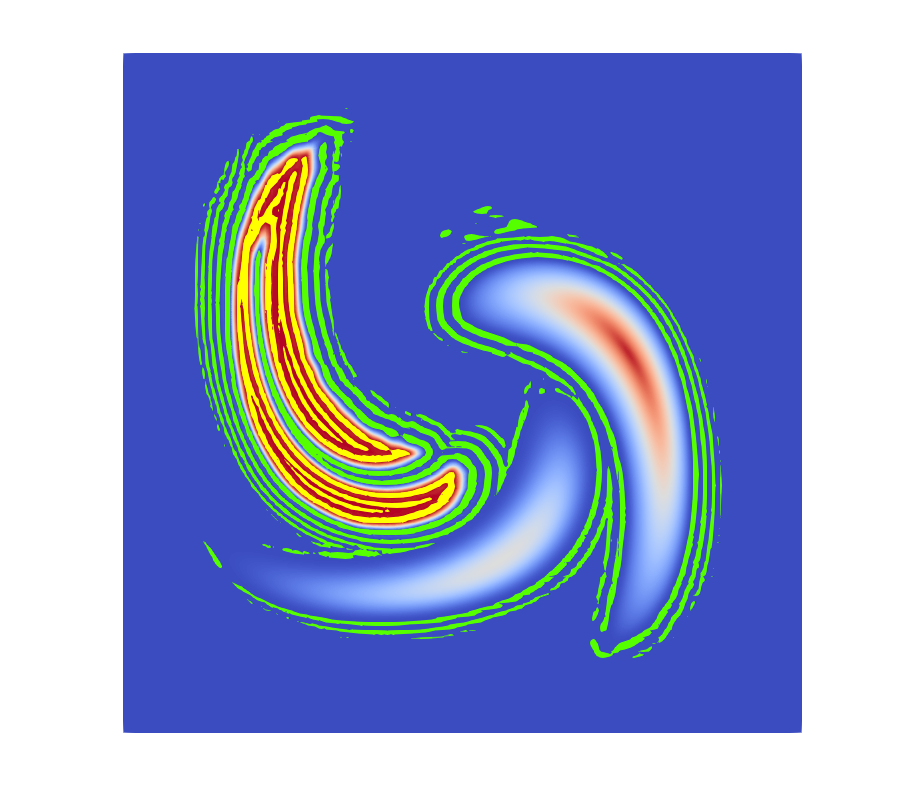}
\end{subfigure}%
\begin{subfigure}{.33\textwidth}
\centering
\includegraphics[trim={0cm 0cm 0cm 0cm},clip,width=0.9\textwidth] {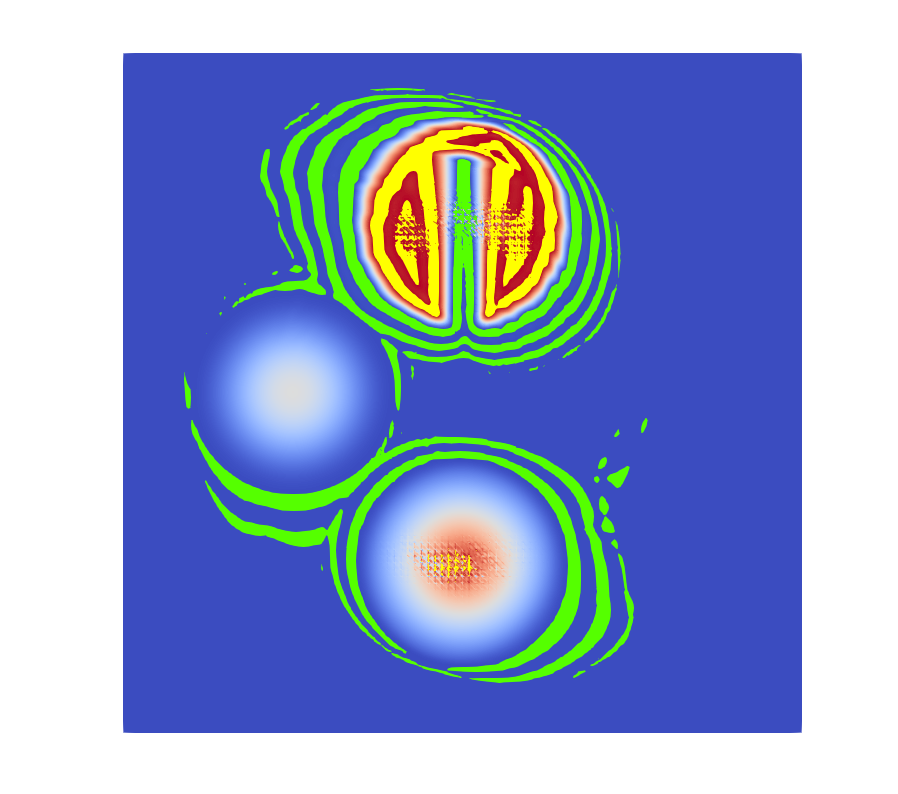}
\end{subfigure}%

\begin{subfigure}{.33\textwidth}
\centering
\rotatebox{90}{\hspace{2em} $\mathcal{L}_2$-$\mathrm{GL}(2)$-$\mathrm{P}$}%
\includegraphics[trim={0cm 0cm 0cm 0cm},clip,width=0.9\textwidth] {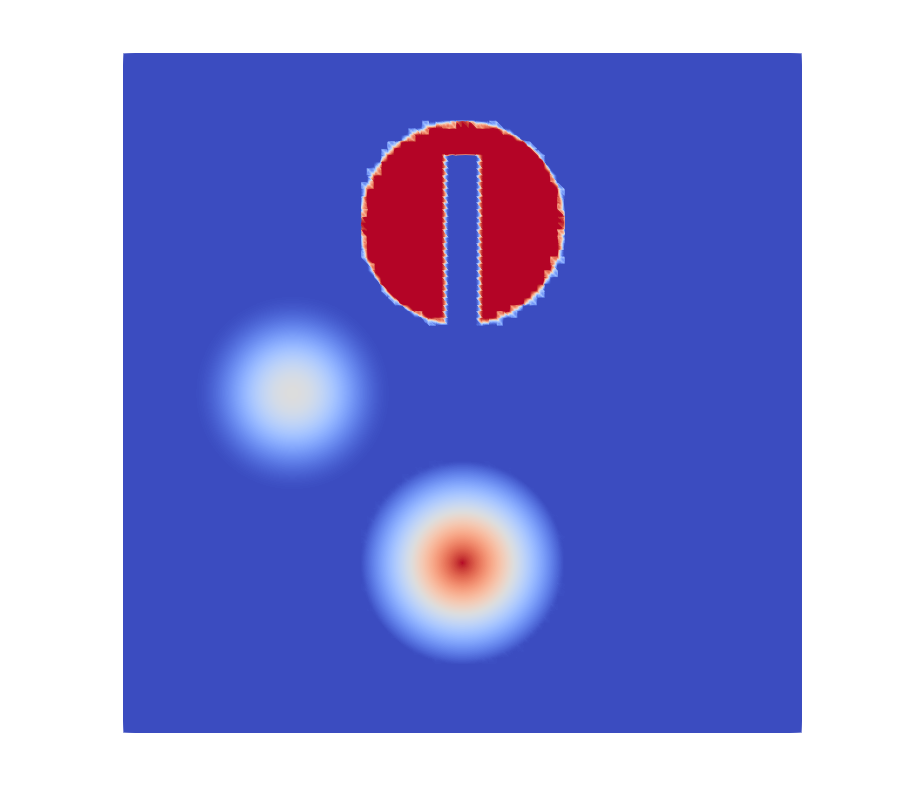}
\end{subfigure}%
\begin{subfigure}{.33\textwidth}
\centering
\includegraphics[trim={0cm 0cm 0cm 0cm},clip,width=0.9\textwidth] {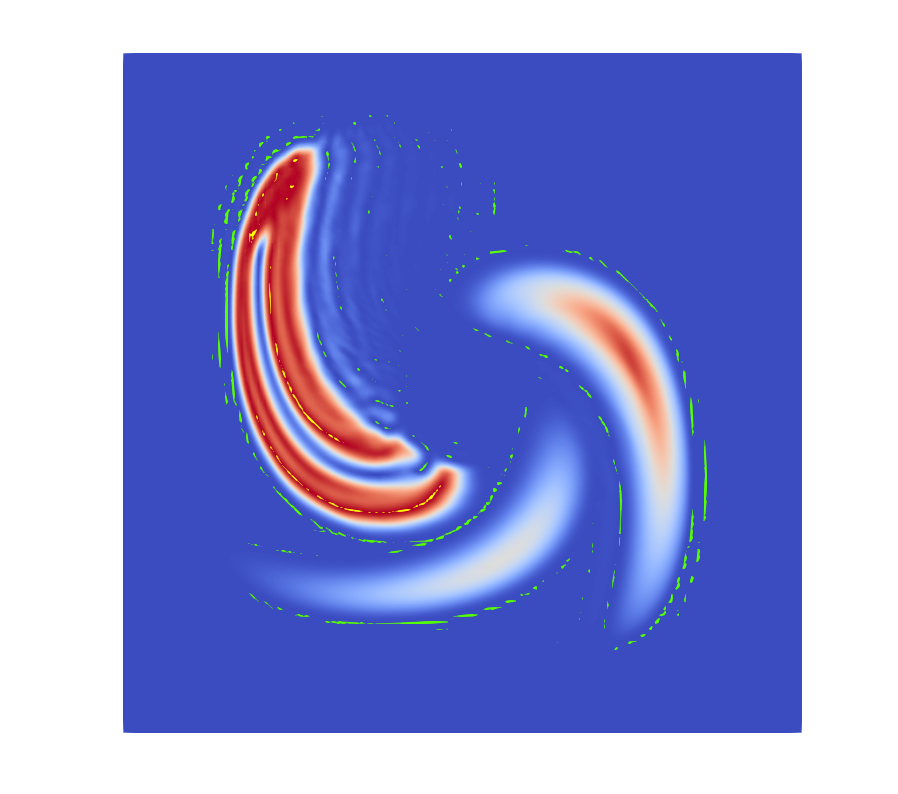}
\end{subfigure}%
\begin{subfigure}{.33\textwidth}
\centering
\includegraphics[trim={0cm 0cm 0cm 0cm},clip,width=0.9\textwidth] {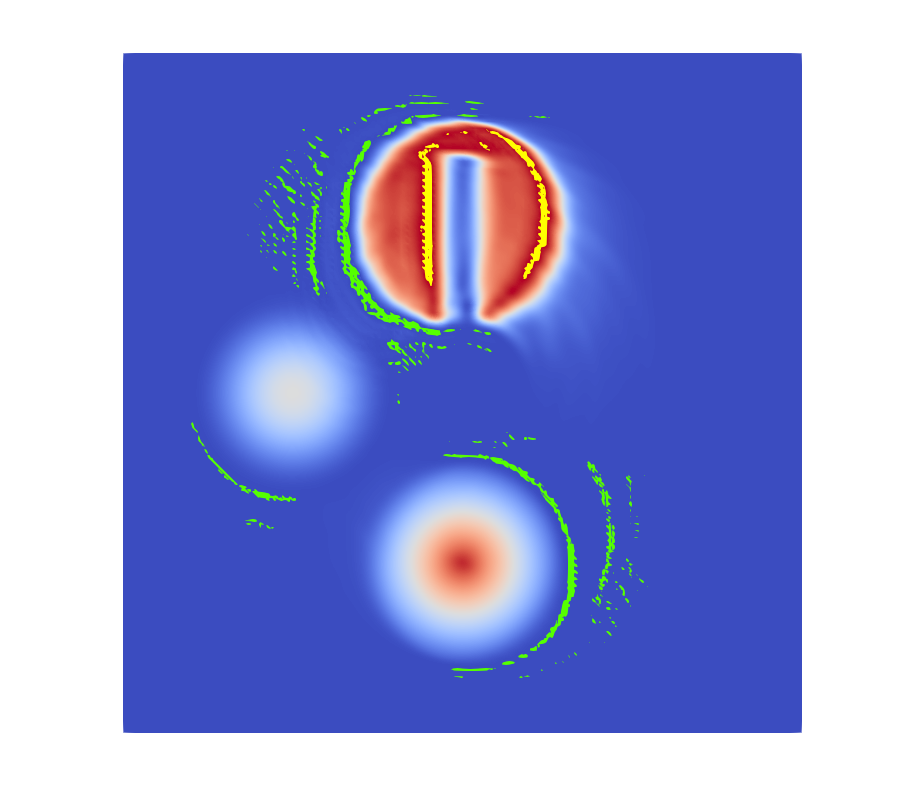}
\end{subfigure}%

\begin{subfigure}{.33\textwidth}
\centering
\rotatebox{90}{\hspace{2em} $\mathcal{L}_2$-$\mathrm{GL}(1)$-$\mathrm{VI}$}%
\includegraphics[trim={0cm 0cm 0cm 0cm},clip,width=0.9\textwidth] {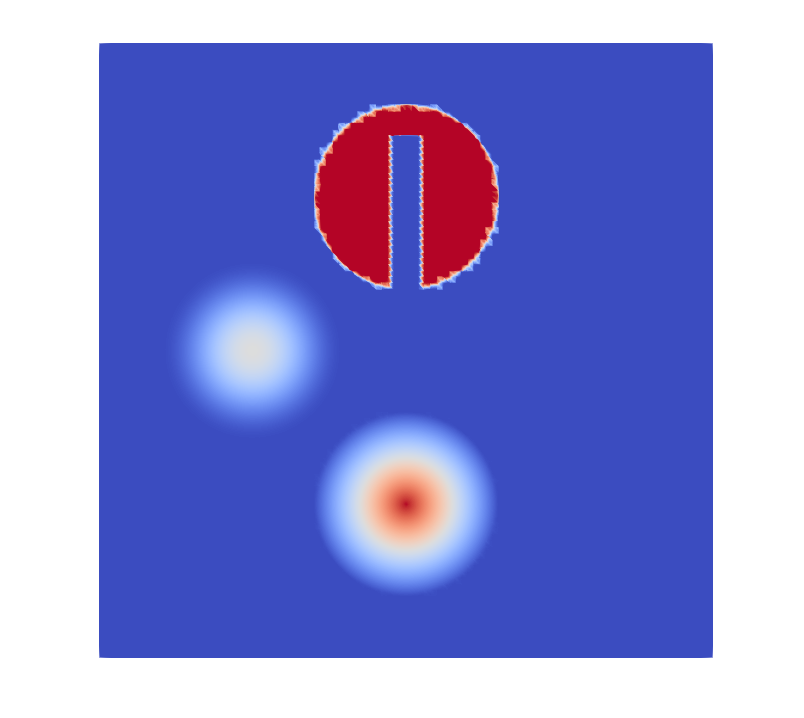}
\subcaption{$t = 0$}
\end{subfigure}%
\begin{subfigure}{.33\textwidth}
\centering
\includegraphics[trim={0cm 0cm 0cm 0cm},clip,width=0.9\textwidth] {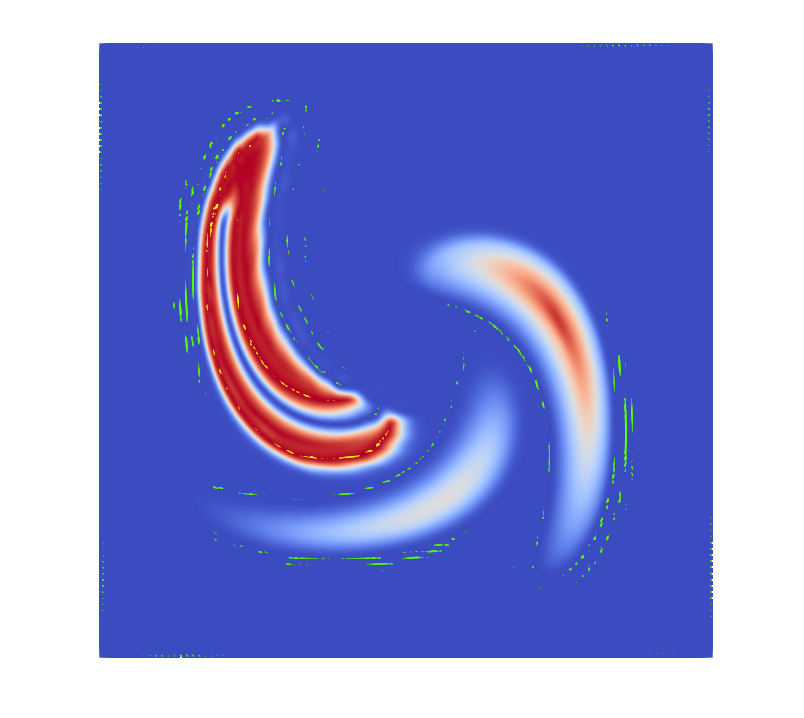}
\subcaption{$t \approx T / 2$}
\end{subfigure}%
\begin{subfigure}{.33\textwidth}
\centering
\includegraphics[trim={0cm 0cm 0cm 0cm},clip,width=0.9\textwidth] {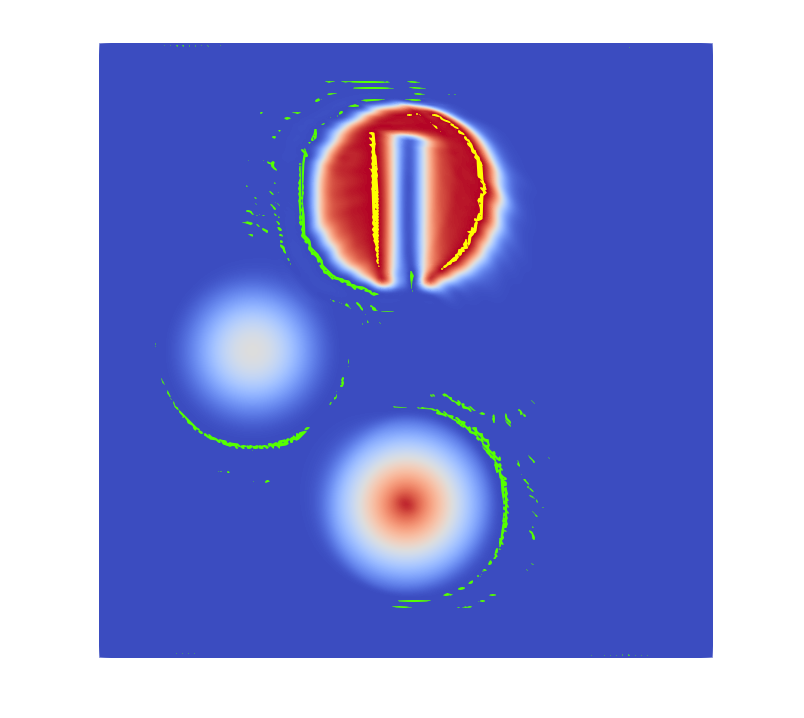}
\subcaption{$t = T$}
\end{subfigure}%

\begin{subfigure}{\textwidth}
\centering
\includegraphics[trim={0cm 0cm 0cm 0cm},clip,width=0.7\textwidth] {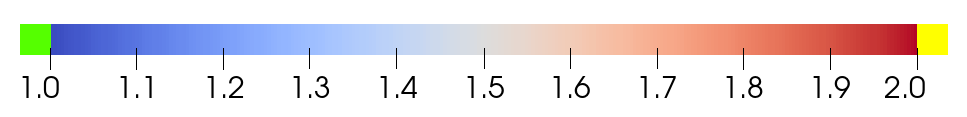}
\end{subfigure}%

\caption{Snapshots of the numerical solution to~\eqref{eq:advec_eq} integrated using (top) $\mathcal{L}_2$-$\mathrm{SSPRK}(3,3)$, (middle) $\mathcal{L}_2$-$\mathrm{GL}(2)$-$\mathrm{P}$, and (bottom) $\mathcal{L}_2$-$\mathrm{GL}(1)$-$\mathrm{VI}$.}
\label{fig:advec_snapshots}
\end{figure}


\subsection{Allen-Cahn}
\label{sec:allen_cahn}
We now turn to the Allen-Cahn equation, which presents a more compelling case for both the monolithic schemes and the uniformly-constrained Bernstein polynomials.
In particular, schemes that allow out-of-bounds values to enter the logarithmic potential (say, at a quadrature point), may simply fail to run to completion.


The Allen-Cahn equation~\eqref{eq:allen_cahn} is the $L^2$ gradient flow of the energy functional
\begin{equation}\label{eq:AC_energy}
  E(u) = \int_\Omega \left(F(u) + \frac{1}{2} \epsilon^2 |\nabla u|^2\right)\mathrm{dx}
\end{equation} 
~\cite{shen2016maximum}. It is easy to see that the solution to the Allen-Cahn equation satisfies the following energy dissipation law:
\begin{equation}
  \frac{\partial}{\partial t} E(u) = -\int_\Omega \left|\frac{\partial u}{\partial t}\right| \mathrm{dx} \leq 0.
\end{equation}
This provides one way to monitor the quality of a numerical solution.

The Flory-Huggins potential will drive phase separation to two critical points in $(-1, 1)$. The specific location corresponds to the wells of the 
potential function. Thus, as $\theta_c$ grows relative to $\theta_0$, the states tend towards the singularities of the logarithm.

To enforce bounds constraints, we choose some $\delta << 1$ such that the minima of the potential function are in the interval $(-1 + \delta, 1 - \delta)$. We then take the 
feasible set to be 
\begin{equation}
  \mathcal{K}_h = \mathcal{J}_h^{\mathcal{B}_k, [- 1 + \delta, 1 - \delta]}.
\end{equation} 
Throughout, we use $\delta = 10^{-8}$. Note that the choice of the Bernstein basis in space is vital. Without proper treatment inside the nonlinear solver, using higher-order Lagrange finite elements, even when constrained at the degrees of 
freedom, can lead to blow up of the logarithm due to sampling at intermediate points. Using the Bernstein basis in space enforces the bounds constraints uniformly, and thereby avoids the singularities. 
The use of a monolithically constrained time stepping method is also critical. Using a projective method requires the computation of an unconstrained approximation before the post-processing procedure enforces the bounds constraints. 
Due to the singularities of the potential, it is often the case that the unconstrained approximation cannot be computed.

\subsubsection{Spinodal Decomposition}

As an initial test we consider a spinodal decomposition. We take the initial condition 
\begin{equation}
  u(x) = 2\textrm{rand}(x) - 1,
\end{equation}
where $\textrm{rand}(x)$ gives a random number between $0$ and $1$. We take $\epsilon = 0.05$, $\theta_0 = 2$, and $\theta_c = 5.0$. We use spatial elements of degree $2$ parameterized using the Bernstein basis, thus taking the discrete feasible set 
\begin{equation}
  \mathcal{K}_h = \mathcal{J}_h^{\mathcal{B}_2, [-1 + \delta, 1 - \delta]}.
\end{equation}
We fix a uniform triangular mesh consisting of $20,000$ elements and simulate to final time $T = 10.0$. We use the $\mathcal{B}_2$-$\textrm{BDF}(4)$-$\textrm{VI}$ method with steps of size $k = 1/2$. Each required starting approximation is computed using $2^4$ steps of the monolithically constrained backward Euler method. 
Snapshots of the evolution are shown in Figure~\ref{fig:AC_spinodal_snaps}, and the associated energy and and maximum and minimum degrees of freedom are show in Figure~\ref{fig:AC_spinodal_data}.
While we do not prove energy dissipation of the method,
we observe this in practice.

\begin{figure}
\begin{subfigure}{.25\textwidth}
\centering
\includegraphics[trim={0cm 0cm 0cm 0cm},clip,width=1.0\textwidth] {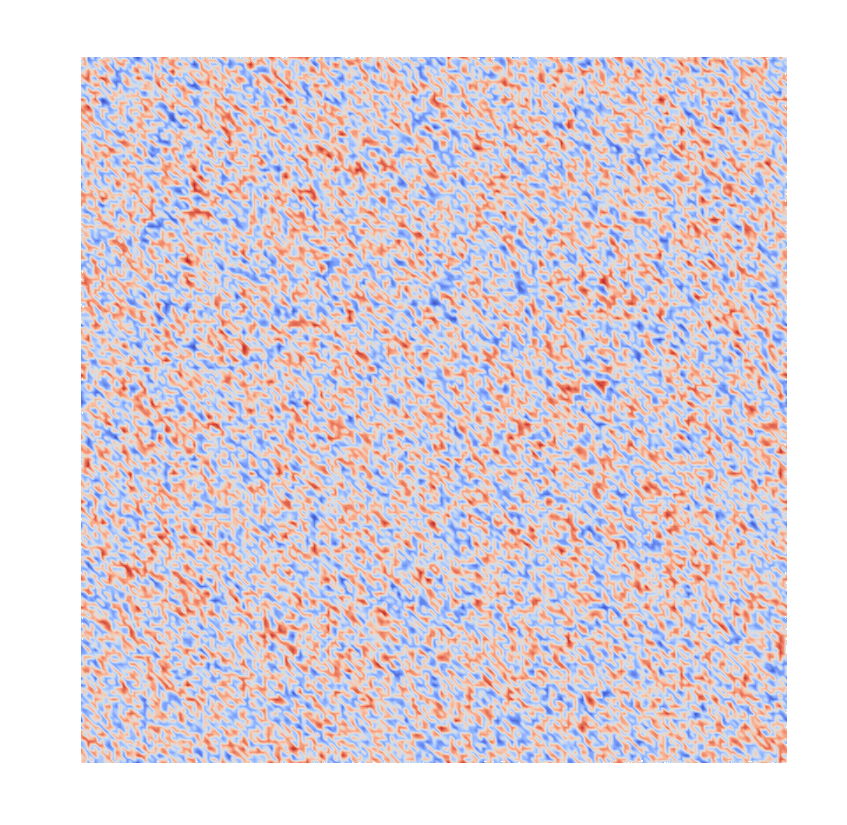}
\subcaption{$t = 0$}
\end{subfigure}%
\begin{subfigure}{.25\textwidth}
\centering
\includegraphics[trim={0cm 0cm 0cm 0cm},clip,width=1.0\textwidth] {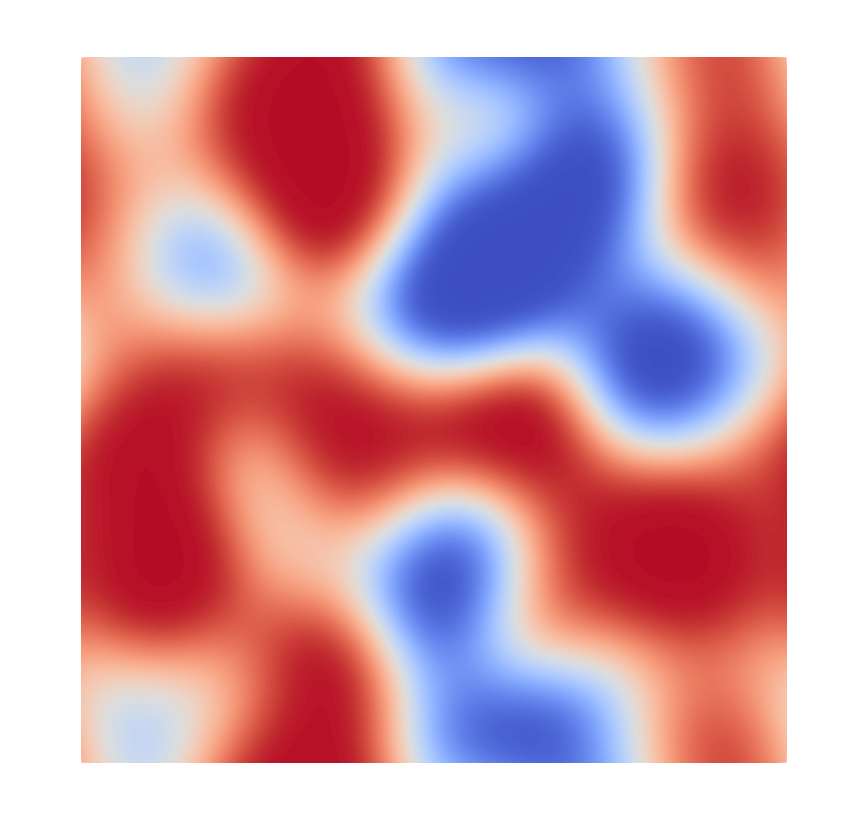}
\subcaption{$t = 1.5$}
\end{subfigure}%
\begin{subfigure}{.25\textwidth}
\centering
\includegraphics[trim={0cm 0cm 0cm 0cm},clip,width=1.0\textwidth] {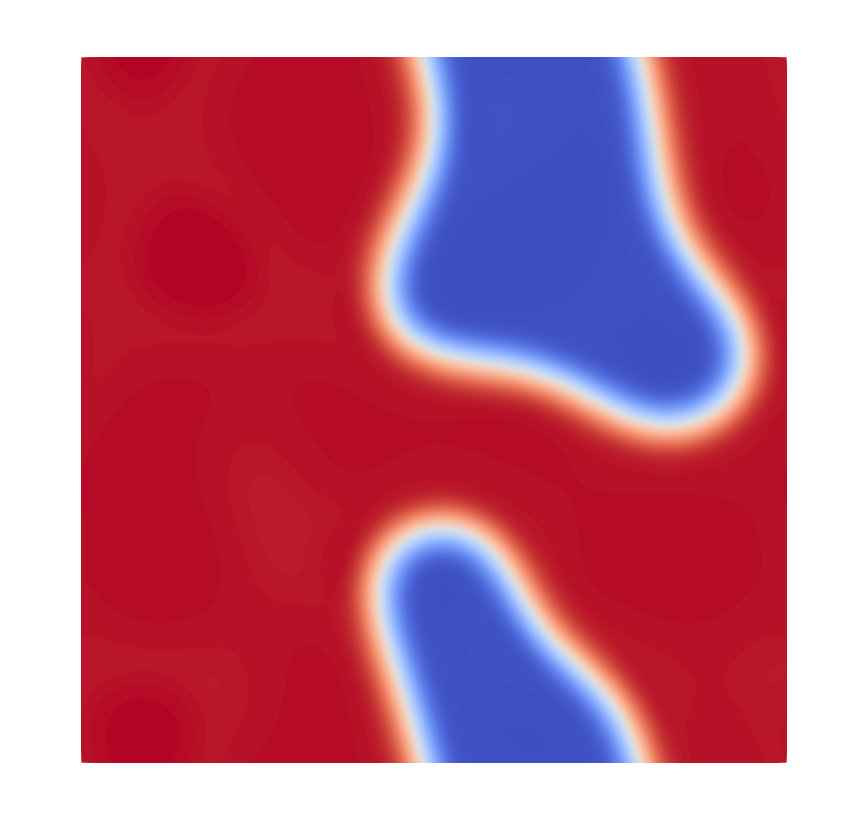}
\subcaption{$t = 3.0$}
\end{subfigure}%
\begin{subfigure}{.25\textwidth}
\centering
\includegraphics[trim={0cm 0cm 0cm 0cm},clip,width=1.0\textwidth] {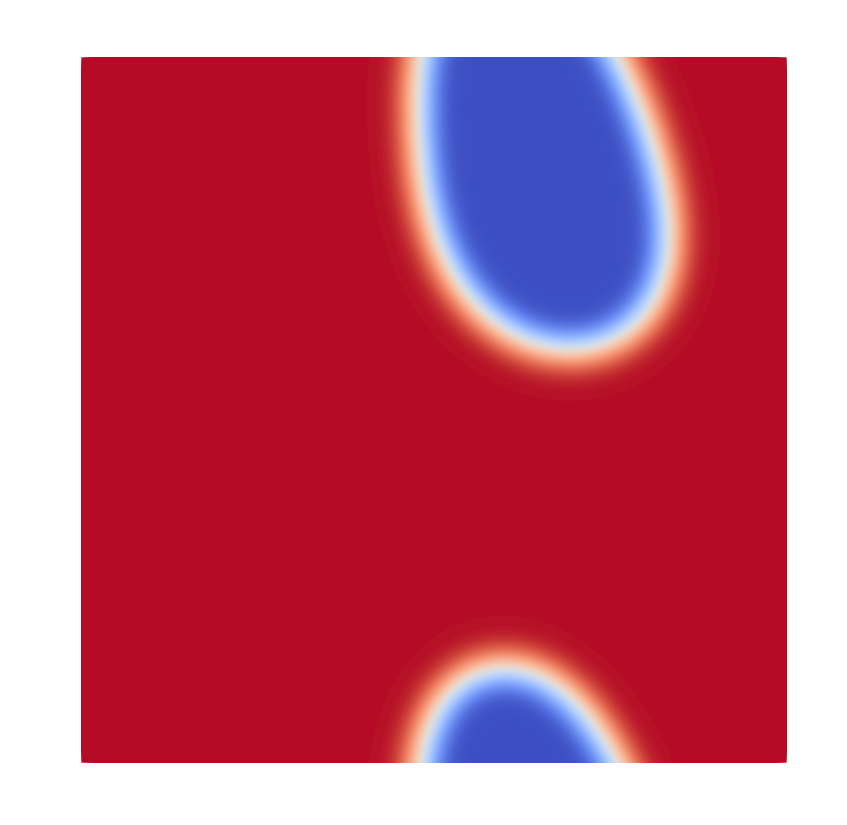}
\subcaption{$t = 10.0$}
\end{subfigure}%

\begin{subfigure}{\textwidth}
\centering
\includegraphics[trim={0cm 0cm 0cm 0cm},clip,width=0.5\textwidth] {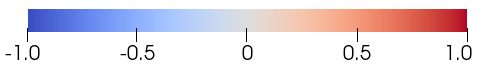}
\end{subfigure}%

\caption{Snapshots of the spinodal decomposition integrated using $\mathcal{B}_2$-$\mathrm{BDF}(2)$-$\mathrm{VI}$.}
\label{fig:AC_spinodal_snaps}
\end{figure}

\pgfplotstableread[col sep=comma]{AC_N100B2BDF4monoTrue_energy.csv}\spinodaldata
\begin{figure}[ht]
  \begin{subfigure}{0.45\textwidth}
  \centering
  \begin{tikzpicture}
  \begin{axis}[width=0.98\textwidth, legend style={at={(0.45, 0.8)}, anchor=west,nodes={scale=0.8, transform shape}, legend columns=1},
    ylabel near ticks,
    xlabel near ticks,
    xtick={0, 2, 4, 6, 8, 10}, 
    xlabel=$t$,
    ymin=-1.25,
    ymax=0.25,
    ylabel={Energy}
        ]

      \addplot[blue,mark=diamond*, mark size=3pt, restrict y to domain=-1:0, unbounded coords=discard, filter discard warning=false, thick] table[x=time, y=energy, col sep=comma] \spinodaldata;
      \addlegendentry[thick] {Energy}

  \end{axis}
  \end{tikzpicture}
\end{subfigure}\hspace{0.01\textwidth}%
\begin{subfigure}{0.45\textwidth}
  \centering
  \begin{tikzpicture}
  \begin{axis}[width=0.98\textwidth, legend style={at={(0.43, 0.5)}, anchor=west,nodes={scale=0.8, transform shape}, legend columns=1},
    ylabel near ticks,
    xlabel near ticks,
    xtick={0, 2, 4, 6, 8, 10}, 
    xlabel=$t$,
    ymin=-1.25,
    ymax=1.25,
    ylabel={Extremal DOFs}
        ]

      \addplot[blue,mark=diamond*, mark size=3pt,  thick] table[x=time, y=maxdof, col sep=comma] \spinodaldata;
      \addlegendentry[thick] {max DOF}
      
      \addplot[brown,mark=*, mark size=3pt, thick] table[x=time, y=mindof, col sep=comma] \spinodaldata;
      \addlegendentry[thick] {min. DOF}

  \end{axis}
  \end{tikzpicture}
\end{subfigure}

\caption{(Left) the energy~\eqref{eq:AC_energy} in the system, and (right) the maximum and minimum degrees of freedom in the approximation.} 
\label{fig:AC_spinodal_data}
\end{figure}
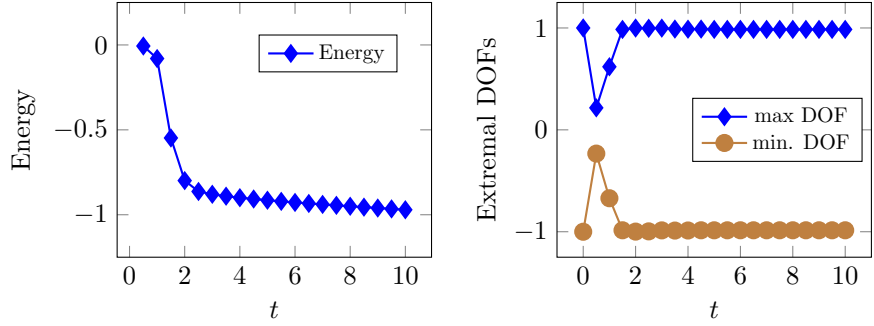

\subsubsection{Anisotropy test}

We now take the bounds-constrained projection of 
\begin{equation}\label{eq:circles}
  u(x) = \begin{cases} 
            1 - \delta, \qquad &\text{if } 0.3 \leq \sqrt{(x - 0.5)^2 + (y - 0.5)^2} \leq 0.4\\
            -1 + \delta, &\text{else}
         \end{cases}
\end{equation}
onto the finite element space as the initial condition. We again take $\epsilon = 0.05$ and $\theta_0 = 2$, but increase to $\theta_c = 8$.

We use an $N\times N$ square mesh, with each square subdivided into two triangles. In Figure~\ref{fig:AC_aniso_data} we plot the energy profiles 
computed using 
$N = 100$, $k = 1/ 8$, and $\mathcal{B}_1$-$\mathrm{BDF}(1)$-$\mathrm{VI}$, $N = 50$, $k = 1/4$ and $\mathcal{B}_2$-$\mathrm{BDF}(2)$-$\mathrm{VI}$, and $N = 25$, $k = 1/2$ and $\mathcal{B}_3$-$\mathrm{BDF}(3)$-$\mathrm{VI}$. Each required starting approximation is computed using $2^s$ steps of the monolithically constrained backward Euler method, with $s$ representing the order of the multistep method. 
We compare against a high-resolution energy profile computed using $\mathcal{B}_2$-$\mathrm{BDF}(2)$-$\mathrm{VI}$ with step size $k = 10^{-3}$ and $N = 100$. In all cases, we see that the energy profile is qualitatively correct. We see that, using higher-order elements and time-stepping methods, better agreement with the high-resolution profile can be obtained with coarser meshes and larger time steps. 
Snapshots of the evolution computed using $N = 50$ and $\mathcal{B}_2$-$\mathrm{BDF}(2)$-$\mathrm{VI}$ are shown in Figure~\ref{fig:AC_aniso_snaps}.

\pgfplotstableread[col sep=comma]{ACex3N100B2BDF2dt_0.001_energy_filtered.csv}\BDFFineData
\pgfplotstableread[col sep=comma]{ACex3N100B1BDF1dt_0.125_energy.csv}\BDFOneData
\pgfplotstableread[col sep=comma]{ACex3N50B2BDF2dt_0.25_energy.csv}\BDFTwoData
\pgfplotstableread[col sep=comma]{ACex3N25B3BDF3dt_0.5_energy.csv}\BDFThreeData

\begin{figure}[ht]
\centering
\begin{subfigure}{0.45\textwidth}
  \centering
\begin{tikzpicture}
\begin{axis}[legend to name=energyLegend, width=0.95\textwidth, legend style={nodes={scale=0.8, transform shape}, column sep=5pt, legend cell align=left, legend columns=2},
  ylabel near ticks,
  xlabel near ticks,
  xtick={0, 5, ..., 35}, 
  xlabel=$t$,
  xmin=-1.0,
  xmax=36.0,
  ymin=-2.65,
  ymax=-1.8,
  ylabel={Energy}
      ]

    \addplot[only marks, blue,mark=diamond, mark size=2.5pt, each nth point=8, restrict y to domain=-3:-1, unbounded coords=discard, filter discard warning=false, thick] table[x=time, y=energy, col sep=comma] \BDFOneData;
    \addlegendentry[thick] {$s = 1,\; N = 100,\; k = 1/8$}
    \addplot[only marks, brown,mark=o, mark size=2.5pt, each nth point=4, restrict y to domain=-3:-1, unbounded coords=discard, filter discard warning=false, thick] table[x=time, y=energy, col sep=comma] \BDFTwoData;
    \addlegendentry[thick] {$s = 2,\; N = 50,\; k = 1/4$}
    \addplot[only marks, darkgray,mark=triangle, mark size=2.5pt, each nth point=2, restrict y to domain=-3:-1, unbounded coords=discard, filter discard warning=false, thick] table[x=time, y=energy, col sep=comma] \BDFThreeData;
    \addlegendentry[thick] {$s = 3,\; N = 25,\; k = 1/2$}
    \addplot[gray, dashed, restrict y to domain=-3:-1, unbounded coords=discard, filter discard warning=false, thick] table[x=time, y=energy, col sep=comma] \BDFFineData;
    \addlegendentry[thick] {High-Res}
\end{axis}
\end{tikzpicture}
\end{subfigure}\hspace{0.01\textwidth}%
\centering
\begin{subfigure}{0.45\textwidth}
\begin{tikzpicture}
\begin{axis}[width=0.95\textwidth,
  ylabel near ticks,
  xlabel near ticks,
  xtick={16.5, 17, ..., 18.5}, 
  xlabel=$t$,
  xmin=16.75,
  xmax=18.75,
  ymin=-2.35,
  ymax=-2.25,
  ylabel={Energy}
      ]

    \addplot[only marks, blue,mark=diamond, mark size=4pt, restrict y to domain=-3:-1, unbounded coords=discard, filter discard warning=false, thick] table[x=time, y=energy, col sep=comma] \BDFOneData;
    \addplot[only marks, brown,mark=o, mark size=4pt, restrict y to domain=-3:-1, unbounded coords=discard, filter discard warning=false, thick] table[x=time, y=energy, col sep=comma] \BDFTwoData;
    \addplot[only marks, darkgray,mark=triangle, mark size=4pt, restrict y to domain=-3:-1, unbounded coords=discard, filter discard warning=false, thick] table[x=time, y=energy, col sep=comma] \BDFThreeData;
    \addplot[gray, dashed, restrict y to domain=-3:-1, unbounded coords=discard, filter discard warning=false, thick] table[x=time, y=energy, col sep=comma] \BDFFineData;
\end{axis}
\end{tikzpicture}
\end{subfigure}%

\begin{subfigure}{1.0\textwidth}
\centering
\ref*{energyLegend}
\end{subfigure}

\caption{Energy profiles of the numerical solution of~\eqref{eq:allen_cahn} subject to the initial condition~\eqref{eq:circles} using $\mathcal{B}_s$-$\mathrm{BDF}(s)$-$\mathrm{VI}$ for $s = 1, 2, 3$. We plot the energy at every whole time on the left and at every time step on the right. } 
\label{fig:AC_aniso_data}
\end{figure}
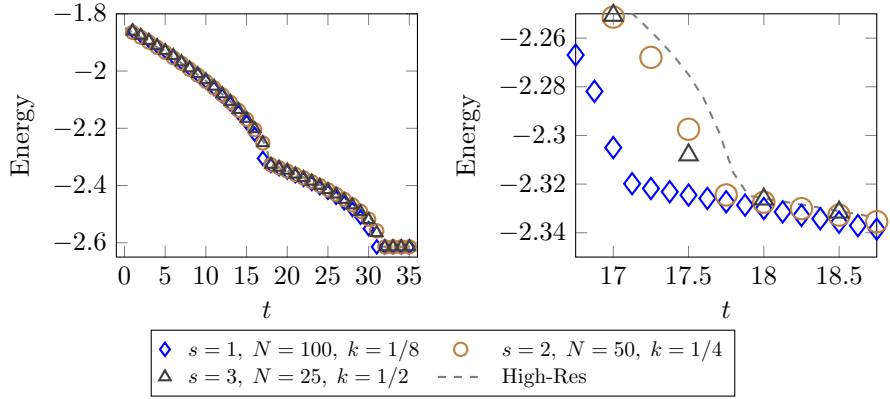

\begin{figure}


\begin{subfigure}{.25\textwidth}
\centering
\includegraphics[trim={4cm 4cm 4cm 4cm},clip,width=1.0\textwidth] {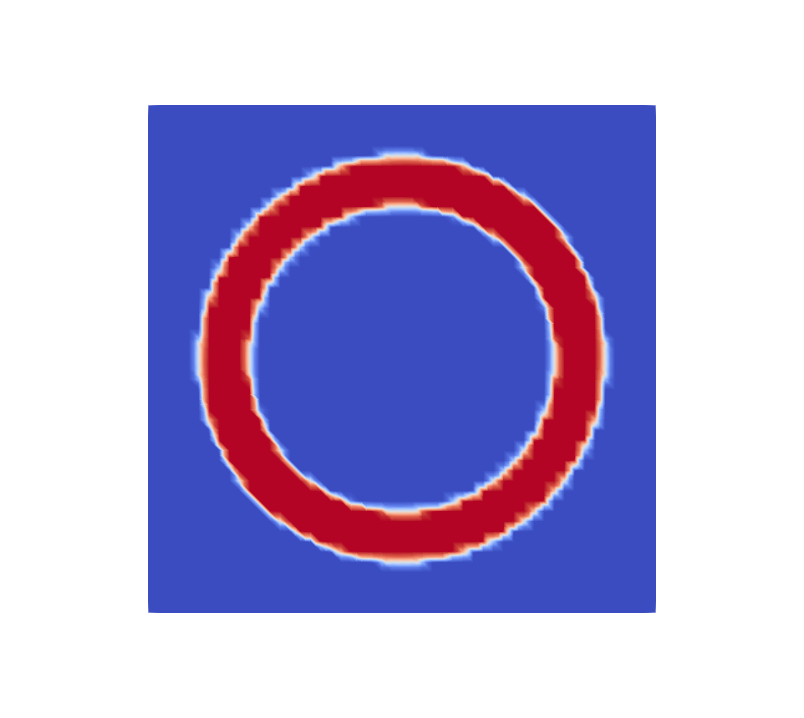}
\subcaption{$t = 0$}
\end{subfigure}%
\begin{subfigure}{.25\textwidth}
\centering
\includegraphics[trim={4cm 4cm 4cm 4cm},clip,width=1.0\textwidth] {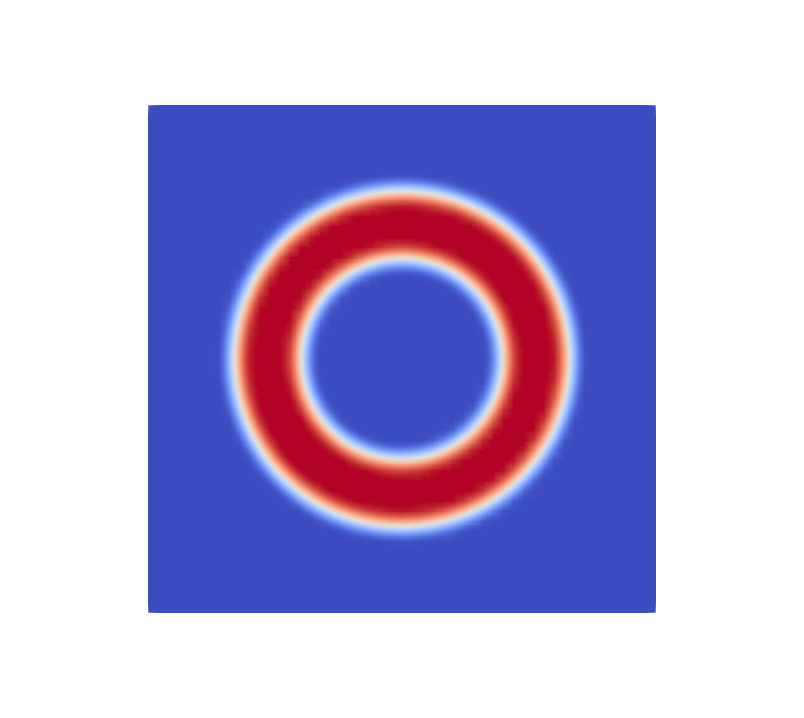}
\subcaption{$t = 10$}
\end{subfigure}%
\begin{subfigure}{.25\textwidth}
\centering
\includegraphics[trim={4cm 4cm 4cm 4cm},clip,width=1.0\textwidth] {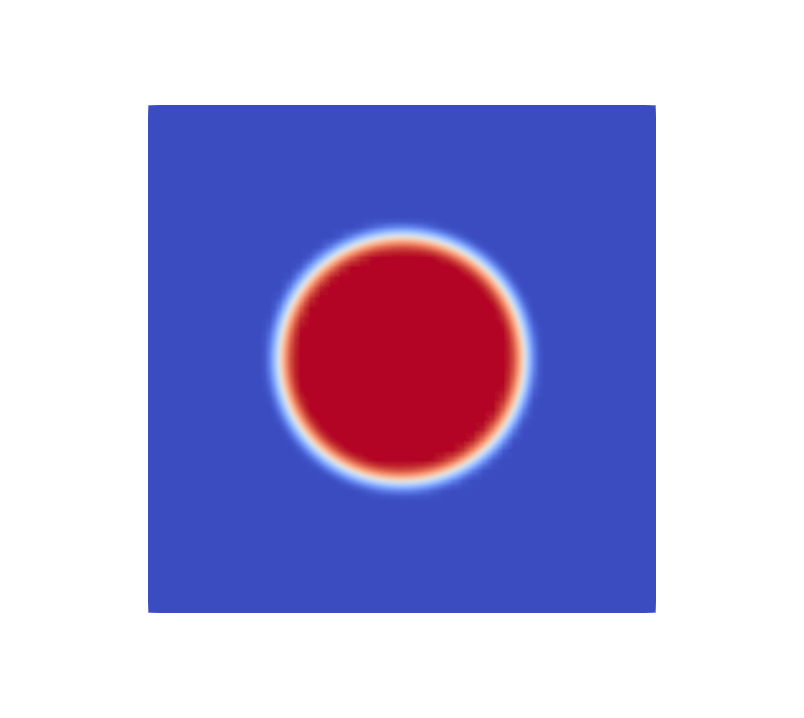}
\subcaption{$t = 20$}
\end{subfigure}%
\begin{subfigure}{.25\textwidth}
\centering
\includegraphics[trim={4cm 4cm 4cm 4cm},clip,width=1.0\textwidth] {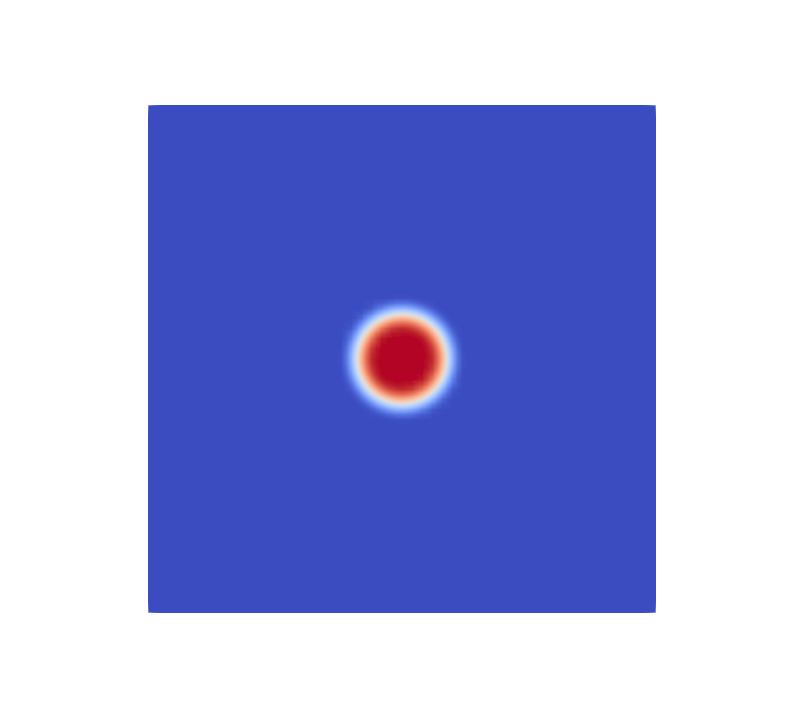}
\subcaption{$t = 30$}
\end{subfigure}%

\begin{subfigure}{\textwidth}
\centering
\includegraphics[trim={0cm 0cm 0cm 0cm},clip,width=0.5\textwidth] {colorbar_AC.png}
\end{subfigure}%

\caption{Snapshots of the evolution with the initial condition~\eqref{eq:circles} integrated using $N = 50$, $k = 1/4$, and $\mathcal{B}_2$-$\mathrm{BDF}(2)$-$\mathrm{VI}$.}
\label{fig:AC_aniso_snaps}
\end{figure}

\section{Conclusions and Future Work}\label{sec:conclusion}

Bounds constraints may be enforced for time-dependent partial differential equations using projective or monolithic techniques. Both of these methods utilize the 
theory of constrained variational inequalities in order to produce approximations which are bounds-constrained (in some sense) at the discrete times.

Projective methods compose a standard time-stepping scheme with a nonlinear projection onto a discrete feasible set.
In our numerical examples, this gives bounds-constrained solutions without lowering the order of accuracy of the approximate solution and allows reuse of existing scalable solvers for the algebraic systems.

For single-stage but possibly multi-step methods, bounds constraints may also be enforced monolithically by recasting the variational equation as a variational inequality.  While this approach is more expensive than the projective methods, it allows for the solution of problems where a nonlinearity is quite sensitive to the bounds constraints.

Also, the methods in this paper focus on enforcement of bounds constraints at the discrete time levels.  For lower-order methods, this is sufficient.
However, higher-order methods may also require enforcement of bounds constraints between those time levels.  In future work, we extend the monolithic approach to the class of implicit collocation-type Runge-Kutta methods where, 
by casting the entire stage-coupled system as a constrained variational inequality, and using a novel reformulation of the system, bounds constraints may be enforced uniformly throughout the temporal domain.
This allows for more refined control of 
the bounds-constraint enforcement, and provides a new class of bounds-constrained integration methods to study.